\documentclass[10pt]{article}
\usepackage{latexsym,amsfonts,amsthm,amsmath,amssymb,cases,authblk,multicol,chngcntr,graphicx,algorithm2e,cite}
\usepackage{enumerate}
\usepackage{graphicx,epstopdf}
\usepackage{color}
\newtheorem{thm}{Theorem}[section]

\newtheorem{prob}{Problem}[section]

\newtheorem{remark}{Remark}[section]
\newtheorem{prop}{Proposition}[section]
\newtheorem{defn}{Definition}[section]
\newtheorem{example}{Example}[section]
\newcounter{nextauthor}
\def\mathrm{\mbox}

\allowdisplaybreaks[4]

\title{{\Large \bf Linear-quadratic Stochastic Stackelberg Differential Games with Affine Constraints}\thanks{This work was supported by the National Natural Science Foundation of China (12171339), the Scientific and Technological Research Program of Chongqing Municipal Education Commission, the grant from Chongqing Technology and Business University (2356004), and  the Fundamental Research Funds for the Central Universities (2682023CX071).}}
\begin{document}
\title{{\Large \bf Linear-quadratic Stochastic Stackelberg Differential Games with Affine Constraints}\thanks{This work was supported by the National Natural Science Foundation of China (12171339), the Scientific and Technological Research Program of Chongqing Municipal Education Commission, the grant from Chongqing Technology and Business University (2356004), and  the Fundamental Research Funds for the Central Universities (2682023CX071).}}
\author[a]{Zhun Gou}
\author[b]{Nan-Jing Huang \thanks{Corresponding author: nanjinghuang@hotmail.com; njhuang@scu.edu.cn}}
\author[a]{Xian-Jun Long}
\author[c]{Jian-Hao Kang}
\affil[a]{\small\it School of Mathematics and Statistics, Chongqing Technology and Business University, Chongqing 400067, P.R. China; Chongqing Key Laboratory of Statistical Intelligent Computing and Monitoring, Chongqing Technology and Business University, Chongqing, 400067, P.R. China}
\affil[b]{Department of Mathematics, Sichuan University, Chengdu, Sichuan 610064, P.R. China}
\affil[c]{School of Mathematics, Southwest Jiaotong University, Chengdu, Sichuan 610031, P.R. China}
\date{}
\maketitle
\begin{center}
\begin{minipage}{5.6in}
\noindent{\bf Abstract.} This paper investigates the non-zero-sum linear-quadratic stochastic Stackelberg differential games with affine constraints, which depend on both the follower's response and the leader's strategy. With the help of the stochastic Riccati equations and the Lagrangian duality theory, the feedback expressions of optimal strategies of the follower and the leader are obtained and the dual problem of the leader's problem is established.  Under the Slater condition, the equivalence is proved between the solutions to the dual problem and the leader's problem, and the KKT condition is also provided for solving the dual problem. Then, the feedback Stackelberg equilibrium is provided for the linear-quadratic stochastic Stackelberg differential games with affine constraints, and a new positive definite condition is proposed for  ensuring the  uniqueness of solutions to the dual problem. Finally, two non-degenerate examples with indefinite coefficients are provided to illustrate and to support our main results.
\end{minipage}
\end{center}

\section{Introduction}
\paragraph{}

\noindent Let $\mathbb{R}^{n\times m}$ be the Euclidean space of  $n\times m$-matrices $\Sigma$ equipped with the norm $\|\Sigma\|_{\mathbb{R}^{n\times m}}=tr^{\frac{1}{2}}(\Sigma^{\top}\Sigma)$ and inner product $\langle\Sigma,\widetilde{\Sigma}\rangle_{\mathbb{R}^{n\times m}}=tr(\Sigma^{\top}\widetilde{\Sigma})$ ($\widetilde{\Sigma}\in \mathbb{R}^{n\times m}$), where $tr(\cdot)$ and $(\cdot)^{\top}$ represent the trace of an $n\times n$-matrix and the  transpose of an $n\times m$-matrix, respectively. Let $I_n$ be the $n\times n$-identity matrix, $0^{m\times n}$ be an $m\times n$-null matrix, $\mathcal{S}^n$ the subset of matrices of $\mathbb{R}^{n\times n}$,  and $\mathcal{S}_{++}^n$ the subset of strictly positive definite  matrices of $\mathcal{S}^n$.

The current paper considers the following controlled linear stochastic differential equation (SDE, for short) in a complete filtered probability space $(\Omega,\mathfrak{F},\mathfrak{F}_t,\mathbb{P})$ satisfying the usual hypothesis:
\begin{equation}\label{SDE1}
\begin{cases}
dX(t)=\left[A(t)X(t)+B_1(t)u_1(t)+B_2(t)u_2(t)\right]dt+C(t)X(t)dB(t), \quad t\in[s,T],\\
X(s)=\xi\in \mathbb{R}^n,
\end{cases}
\end{equation}
where $[s,T]$ denotes a finite time duration with $T>s \geq0$; $X(\cdot)$ is the $\mathbb{R}^n$-valued state variable; $u_1(\cdot)$ and $u_2(\cdot)$ are the strategies ($\mathbb{R}^m$-valued control variables) of the follower and the leader, respectively; $\{B(t)\}_{t\in[0,T]}$ is a standard one-dimensional Brownian motion which generates the right-continuous and increasing $\sigma$-algebra $\mathcal{F}=\{\mathcal{F}_t\}_{t\in[0,T]}$; $A,C:[0,T]\times\Omega\rightarrow\mathbb{R}^{n\times n}$, $B_1,B_2\in[0,T]\times\Omega\rightarrow\mathbb{R}^{n\times m}$ are all bounded and $\mathcal{F}$-progressively measurable coefficients. For any given Euclidean space $\mathcal{H}$ and $p,q\in[1,+\infty)$, we make use of the following notations throughout this paper.
\begin{itemize}
\item $\mathcal{L}^p([s,T],\mathcal{H})$: the space of all measurable functions $f:[s,T]\rightarrow \mathcal{H}$ with $\int_s^T\|f(t)\|^p_{\mathcal{H}}dt<\infty$.

\item $\mathcal{L}^p_{\mathcal{F}_{t}}(\Omega,\mathcal{H})$: the space of all $\mathcal{F}_{t}$-measurable random variables $\eta:\Omega\rightarrow \mathcal{H}$ with $\mathbb{E}[\|\eta\|_{\mathcal{H}}^p]<\infty$.

\item $\mathcal{L}^p_{\mathcal{F}}(\Omega,C([s,T],\mathcal{H}))$: the Banach space of all $\mathcal{F}$-progressively measurable  processes $X:[s,T]\times \Omega\rightarrow \mathcal{H}$ with
$\mathbb{E}\left[\sup \limits_{t\in[s,T]}\|X(t)\|^p_{\mathcal{H}}\right]<\infty$.

\item  $\mathcal{L}^{\infty}_{\mathcal{F}}(\Omega,\mathcal{L}^p([s,T],\mathcal{H}))$: the Banach space of all $\mathcal{F}$-progressively measurable  processes $X:[s,T]\times \Omega\rightarrow \mathcal{H}$  with
$\mathop{\mbox{esssup}} \limits_{\omega\in \Omega} \left(\int_s^T\|X(t)\|_{\mathcal{H}}^pdt\right)<\infty$.

\item  $\mathcal{L}^{\infty}_{\mathcal{F}}(\Omega,C([s,T],\mathcal{H}))$: the Banach space of all $\mathcal{F}$-progressively measurable  processes $X:[s,T]\times \Omega\rightarrow \mathcal{H}$  with
$\mathop{\mbox{esssup}} \limits_{\omega\in \Omega} \left(\sup \limits_{t\in[s,T]}\|X(t)\|_{\mathcal{H}}\right)<\infty$.

\item $\mathcal{L}^q_{\mathcal{F}}(\Omega,\mathcal{L}^p([s,T],\mathcal{H}))$: the space of all $\mathcal{F}$-progressively measurable processes $X:[s,T]\times \Omega\rightarrow \mathcal{H}$  with $\mathbb{E}\left(\left[\int_s^T\|X(t)\|_{\mathcal{H}}^pdt\right]^{\frac{q}{p}}\right)<\infty$.

\end{itemize}
We denote $\mathcal{L}^p_{\mathcal{F}}(\Omega,\mathcal{L}^p([s,T],\mathcal{H}))=\mathcal{L}^p_{\mathcal{F}}([s,T],\mathcal{H})$. Clearly, $\mathcal{L}^2_{\mathcal{F}_{t}}(\Omega,\mathcal{H})$ (resp. $\mathcal{L}^2_{\mathcal{F}}([s,T],\mathcal{H})$) is a Hilbert space equipped with the norm $\|\cdot\|_{\mathcal{L}^2_{\mathcal{F}_{t}}(\Omega,\mathcal{H})}=\mathbb{E}^{\frac{1}{2}}[\|\cdot\|_{\mathcal{H}}^2]$ (resp. $\|\cdot\|_{\mathcal{L}^2_{\mathcal{F}}([s,T],\mathcal{H})}=\mathbb{E}^{\frac{1}{2}}\left[\int_s^T|\cdot(t)|_{\mathcal{H}}^2dt\right]$) and inner product $\langle\cdot,\cdot\rangle_{\mathcal{L}^2_{\mathcal{F}_{t}(\Omega,\mathcal{H})}}$ (resp. $\langle\cdot,\cdot\rangle_{\mathcal{L}^2_{\mathcal{F}}([s,T],\mathcal{H})}$). Without loss of generality, we denote $\langle\cdot,\cdot\rangle$ (resp. $|\cdot|$) the inner product (resp. the norm) in any Euclidean space.  And for notational convenience, we will frequently suppress the $t$-dependence of a process involved in an equation or an integral.

The Stackelberg game, also known as the leader-follower game, was first introduced by Stackelberg \cite{von2010market} to obtain optimal strategies in competitive economics.  There are usually two players in such a game. One player acts as the leader (she) while the other behaves as the follower (he). After the leader announcing her strategy $u_2$, the follower reacts to it by optimising his cost functional with $u_1^*(u_2)$. Then, the leader would like to seek a strategy $\widehat{u}_2$ to optimise her cost functional based on the follower's best response. The best strategy of the leader together with the best response of the follower $(u_1^*(\widehat{u}_2),\widehat{u}_2)$ is called the Stackelberg equilibrium. The leader-follower structure of the Stackelberg game is suitable for modelling many practical problems arising in various fields such as newsvendor problems \cite{oksendal2013stochastic},  supply chain problems \cite{zou2020stackelberg},  advertising problems \cite{chutani2014feedback,meng2023stochastic}, energy problems \cite{yan2024stackelberg}, principal-agent/optimal contract problems \cite{cvitanic2013contract}, just mention a few.
For practical reasons, one may further consider  constraints on the states and the strategies of the leaders in  Stackelberg games. For example, Sanguanpuak et al. \cite{sanguanpuak2019radio} formulated a Stackelberg game model with  latency constraints in SBSs, and gave its Stackelberg equilibrium. Matamala and Feijoo \cite{matamala2021two} proposed a  stochastic Stackelberg game model for microgrid operation with probabilistic constraints for renewable energy generation uncertainty, and obtained the microgrids operational decision. Qin et al. \cite{qin2024integral} proposed an integral reinforcement learning-based dynamic event-triggered safety control scheme to tackle the multiplayer Stackelberg-Nash games  in continuous-time nonlinear systems with time-varying state constraints.

The linear-quadratic stochastic Stackelberg differential game (LQ-SSDG, for short), as one of the most important games in the dynamic Stackelberg game theory, has been widely investigated. It is well-known that the LQ-SSDG is closely related to two optimal control problems, i.e., the follower's problem can be described as an optimal control problem governed by the SDE, while the leader's turns out to be an optimal control problem governed by the fully coupled forward-backward stochastic differential equation (FBSDE, for short). Owing to the nice mathematical structure of LQ-SSDGs, researchers are able to express the Stackelberg equilibria as functions of the state variable, which are called the feedback Stackelberg equilibria. Till now, researchers have obtained the feedback Stackelberg equilibria for various type LQ-SSDGs (see \cite{gou2023linear, moon2022linear, moon2023linear, wu2024zero, xiang2024stochastic, zheng2020linear} and references therein).  Similarly as in Stackelberg games,  we can consider constraints on the states and the strategies of the leaders in the LQ-SSDGs.  It is well-known that  stochastic linear-quadratic control problems are closely related to  LQ-SSDGs. Various type constraints have been considered in stochastic linear-quadratic control problems,  such as linear equality constraints \cite{zhang2023stochastic}, linear inequality constraints \cite{Wu2020On}, quadratic constraints \cite{lim1999stochastic}, terminal state constraints \cite{bank2018linear}, expectation constraints \cite{chow2020dynamic} and affine constraints \cite{gou2024stochastic}. Also, some type constraints have been considered in LQ-SSDGs. For example, Zhang and Zhang \cite{zhang2021global} investigated SSDGs with pointwise constraint, obtained the Stackelberg equilibrium by establishing the Pontryagin  maximum principle for the leader's optimal strategy and applied the results to LQ-SSDGs.  Feng and Hu \cite{feng2022backward} studied the backward LQ-SSDGs with both pointwise constraints and  affine constraints, and gave the Stackelberg equilibrium by some coupled backward-forward stochastic differential equations with mixed initial-terminal conditions and nonlinear projection operator. For  Stackelberg dynamic games with nonlinear mixed equality and inequality constraints, we refer readers to \cite{li2024computation}. Recently, Gou et al. \cite{gou2024stochastic} showed that many expectation-type constraints including quadratic constraints and risk constraints in stochastic linear-quadratic control problems can be approximately captured by finite many affine constraints. Thus a natural question is: ``Does this hold true in the case of LQ-SSDGs?" The following example gives a positive answer.
\begin{example}\label{+18}
For bounded and $\mathcal{F}$-progressively measurable processes $\widetilde{D},\widetilde{G}:[0,T]\times\Omega\rightarrow\mathcal{S}^n$ and
$\widetilde{E_1},\widetilde{E_2}:[0,T]\times\Omega\rightarrow\mathcal{S}^m$,  consider the quadratic functional $\Theta: U_1\rightarrow [0,+\infty)$ defined by
\begin{align*}
\Theta(u_2)=&\langle X^*,\widetilde{D}X^*\rangle_{\mathcal{L}^2_{\mathcal{F}}([s,T],\mathbb{R}^n)}+\langle u^*_1,\widetilde{E_1}u^*_1\rangle_{\mathcal{L}^2_{\mathcal{F}}([s,T],\mathbb{R}^m)} +\langle u_2,\widetilde{E_2}u_2\rangle_{\mathcal{L}^2_{\mathcal{F}}([s,T],\mathbb{R}^m)}+\langle X^*(T),\widetilde{G}X^*(T)\rangle_{\mathcal{L}^2_{\mathcal{F}_{T}}(\Omega,\mathbb{R}^n)},
\end{align*}
where
$
U_1=\left\{u_2\in\mathcal{L}^2_{\mathcal{F}}([s,T],\mathbb{R}^m)\big|\Theta(u_2)\leq a_0\right\},
$
the triple $(X^*,u^*_1,u_2)$ is determined by \eqref{SDE1}, $u^*_1=u^*_1(u_2)$ is the follower's response which depends on the leader's announced strategy $u_2$, and  $a_0$ is a given positive constant. Then one can easily show that $U_1$ is closed and convex in $\mathcal{L}^2_{\mathcal{F}}([s,T],\mathbb{R}^m)$. On the other hand, let
\begin{align*}
U_2=\Bigg\{&u\in\mathcal{L}^2_{\mathcal{F}}([s,T],\mathbb{R}^m)\big|\mbox{for $(X^*,u^*_1,u_2)$ governed by \eqref{SDE1} and any}\;\\
&(\widetilde{\alpha},\widetilde{\beta}_1,\widetilde{\beta}_2,\widetilde{\gamma})\in\mathcal{L}^2_{\mathcal{F}}([s,T],\mathbb{R}^n)\times\mathcal{L}^2_{\mathcal{F}}([s,T],\mathbb{R}^m)
\times\mathcal{L}^2_{\mathcal{F}}([s,T],\mathbb{R}^m)\times\mathcal{L}^2_{\mathcal{F}_{T}}(\Omega,\mathbb{R}^n)\;\mbox{such that}\\
&\langle\widetilde{D}\widetilde{\alpha},\widetilde{\alpha}\rangle_{\mathcal{L}^2_{\mathcal{F}}([s,T],\mathbb{R}^n)}+\langle \widetilde{E_1}\widetilde{\beta}_1,\widetilde{\beta}_1\rangle_{\mathcal{L}^2_{\mathcal{F}}([s,T],\mathbb{R}^m)}+\langle \widetilde{E_2}\widetilde{\beta}_2,\widetilde{\beta}_2\rangle_{\mathcal{L}^2_{\mathcal{F}}([s,T],\mathbb{R}^m)}+\langle\widetilde{G}\widetilde{\gamma},\widetilde{\gamma}\rangle_{\mathcal{L}^2_{\mathcal{F}_{T}}(\Omega,\mathbb{R}^n)}= a_0,\\
&\mbox{the following inquality holds:}\\
&\langle \widetilde{\alpha},\widetilde{D}X^*\rangle_{\mathcal{L}^2_{\mathcal{F}}([s,T],\mathbb{R}^n)}+\langle \widetilde{\beta}_1,\widetilde{E_1}u^*_1\rangle_{\mathcal{L}^2_{\mathcal{F}}([s,T],\mathbb{R}^m)}+
\langle \widetilde{\beta}_2,\widetilde{E_2}u_2\rangle_{\mathcal{L}^2_{\mathcal{F}}([s,T],\mathbb{R}^m)}+\langle \widetilde{\gamma},\widetilde{G}X^*(T)\rangle_{\mathcal{L}^2_{\mathcal{F}_{T}}(\Omega,\mathbb{R}^n)}\leq a_0\Bigg\}.
\end{align*}
Then for any $(\widetilde{\alpha},\widetilde{\beta}_1,\widetilde{\beta}_2,\widetilde{\gamma})\in\mathcal{L}^2_{\mathcal{F}}([s,T],\mathbb{R}^n)\times\mathcal{L}^2_{\mathcal{F}}([s,T],\mathbb{R}^m)
\times\mathcal{L}^2_{\mathcal{F}}([s,T],\mathbb{R}^m)\times\mathcal{L}^2_{\mathcal{F}_{T}}(\Omega,\mathbb{R}^n)$ satisfying
\begin{equation}\label{+17}
\langle\widetilde{D}\widetilde{\alpha},\widetilde{\alpha}\rangle_{\mathcal{L}^2_{\mathcal{F}}([s,T],\mathbb{R}^n)}+\langle \widetilde{E_1}\widetilde{\beta}_1,\widetilde{\beta}_1\rangle_{\mathcal{L}^2_{\mathcal{F}}([s,T],\mathbb{R}^m)}+\langle \widetilde{E_2}\widetilde{\beta}_2,\widetilde{\beta}_2\rangle_{\mathcal{L}^2_{\mathcal{F}}([s,T],\mathbb{R}^m)}+\langle\widetilde{G}\widetilde{\gamma},\widetilde{\gamma}\rangle_{\mathcal{L}^2_{\mathcal{F}_{T}}(\Omega,\mathbb{R}^n)}= a_0,
\end{equation}
we have $U_1\subseteq U_2$ because for any $u_2\in U_1$,
\begin{align*}
0\leq &\langle X^*-\widetilde{\alpha},\widetilde{D}(X^*-\widetilde{\alpha})\rangle_{\mathcal{L}^2_{\mathcal{F}}([s,T],\mathbb{R}^n)}+\langle u^*_1-\widetilde{\beta}_1,\widetilde{E_1}(u^*_1-\widetilde{\beta}_1)\rangle_{\mathcal{L}^2_{\mathcal{F}}([s,T],\mathbb{R}^m)}\\
 &\mbox{}+\langle u_2-\widetilde{\beta}_2,\widetilde{E_2}(u_2-\widetilde{\beta}_2)\rangle_{\mathcal{L}^2_{\mathcal{F}}([s,T],\mathbb{R}^m)}+\langle X^*(T)-\widetilde{\gamma},\widetilde{G}(X^*(T)-\widetilde{\gamma})\rangle_{\mathcal{L}^2_{\mathcal{F}_{T}}(\Omega,\mathbb{R}^n)}\\
\leq &\Theta(u)+\langle\widetilde{D}\widetilde{\alpha},\widetilde{\alpha}\rangle_{\mathcal{L}^2_{\mathcal{F}}([s,T],\mathbb{R}^n)}+\langle \widetilde{E_1}\widetilde{\beta}_1,\widetilde{\beta}_1\rangle_{\mathcal{L}^2_{\mathcal{F}}([s,T],\mathbb{R}^m)}+\langle \widetilde{E_2}\widetilde{\beta}_2,\widetilde{\beta}_2\rangle_{\mathcal{L}^2_{\mathcal{F}}([s,T],\mathbb{R}^m)}
+\langle\widetilde{G}\widetilde{\gamma},\widetilde{\gamma}\rangle_{\mathcal{L}^2_{\mathcal{F}_{T}}(\Omega,\mathbb{R}^n)}\\
&\mbox{}-2\left[\langle \widetilde{\alpha},\widetilde{D}X^*\rangle_{\mathcal{L}^2_{\mathcal{F}}([s,T],\mathbb{R}^n)}+\langle \widetilde{\beta}_1,\widetilde{E_1}u^*_1\rangle_{\mathcal{L}^2_{\mathcal{F}}([s,T],\mathbb{R}^m)}+
\langle \widetilde{\beta}_2,\widetilde{E_2}u_2\rangle_{\mathcal{L}^2_{\mathcal{F}}([s,T],\mathbb{R}^m)}+\langle \widetilde{\gamma},\widetilde{G}X^*(T)\rangle_{\mathcal{L}^2_{\mathcal{F}_{T}}(\Omega,\mathbb{R}^n)}\right]\\
\leq &2a_0-2\left[\langle \widetilde{\alpha},\widetilde{D}X^*\rangle_{\mathcal{L}^2_{\mathcal{F}}([s,T],\mathbb{R}^n)}+\langle \widetilde{\beta}_1,\widetilde{E_1}u^*_1\rangle_{\mathcal{L}^2_{\mathcal{F}}([s,T],\mathbb{R}^m)}+
\langle \widetilde{\beta}_2,\widetilde{E_2}u_2\rangle_{\mathcal{L}^2_{\mathcal{F}}([s,T],\mathbb{R}^m)}+\langle \widetilde{\gamma},\widetilde{G}X^*(T)\rangle_{\mathcal{L}^2_{\mathcal{F}_{T}}(\Omega,\mathbb{R}^n)}\right].
\end{align*}
On the other hand, for any given $u_2\in U_2$, the state $X^*\in\mathcal{L}^2_{\mathcal{F}}([s,T],\mathbb{R}^n)$ is uniquely determined. Letting $(\widetilde{\alpha},\widetilde{\beta}_1,\widetilde{\beta}_2,\widetilde{\gamma})
=\left(\sqrt{\frac{a_0}{\Theta(\widetilde{u})}}X^*,\sqrt{\frac{a_0}{\Theta(\widetilde{u})}}u_1^*,\sqrt{\frac{a_0}{\Theta(\widetilde{u})}}u_2,
\sqrt{\frac{a_0}{\Theta(\widetilde{u})}}X^*(T)\right)$ which satisfies \eqref{+17}, one has
$$
\langle \widetilde{\alpha},\widetilde{D}X^*\rangle_{\mathcal{L}^2_{\mathcal{F}}([s,T],\mathbb{R}^n)}+\langle \widetilde{\beta}_1,\widetilde{E_1}u^*_1\rangle_{\mathcal{L}^2_{\mathcal{F}}([s,T],\mathbb{R}^m)}+
\langle \widetilde{\beta}_2,\widetilde{E_2}u_2\rangle_{\mathcal{L}^2_{\mathcal{F}}([s,T],\mathbb{R}^m)}+\langle \widetilde{\gamma},\widetilde{G}X^*(T)\rangle_{\mathcal{L}^2_{\mathcal{F}_{T}}(\Omega,\mathbb{R}^n)}
=\sqrt{a_0\Theta(u_2})\leq a_0,
$$
which means that $\Theta(u_2)\leq a_0$ and so $u_2\in U_1$.  This implies that $U_2\subseteq U_1$, which leads to $U_1=U_2$. Let
\begin{align*}
U^{p}=\Bigg\{&u_2\in\mathcal{L}^2_{\mathcal{F}}([s,T],\mathbb{R}^m)\big|\mbox{for $(X^*,u^*_1,u_2)$ governed \eqref{SDE1} and any}\;\\
&(\widetilde{\alpha}_i,\widetilde{\beta}_{1,i},\widetilde{\beta}_{2,i},\widetilde{\gamma}_i)\in\mathcal{L}^2_{\mathcal{F}}([s,T],\mathbb{R}^n)\times\mathcal{L}^2_{\mathcal{F}}([s,T],\mathbb{R}^m)
\times\mathcal{L}^2_{\mathcal{F}}([s,T],\mathbb{R}^m)\times\mathcal{L}^2_{\mathcal{F}_{T}}(\Omega,\mathbb{R}^n)\;\mbox{such that}\\
&\mathbb{P}\left\{(\widetilde{\alpha}_i,\widetilde{\beta}_{1,i},\widetilde{\beta}_{2,i},\widetilde{\gamma}_i)=(\widetilde{\alpha}_j,\widetilde{\beta}_{1,j},\widetilde{\beta}_{2,j},\widetilde{\gamma}_j),\;i\neq j,\;i,j=1,2,\cdots,p\right\}=0\;\mbox{and}\\
&\langle\widetilde{D}\widetilde{\alpha},\widetilde{\alpha}\rangle_{\mathcal{L}^2_{\mathcal{F}}([s,T],\mathbb{R}^n)}+\langle \widetilde{E_1}\widetilde{\beta}_1,\widetilde{\beta}_1\rangle_{\mathcal{L}^2_{\mathcal{F}}([s,T],\mathbb{R}^m)}+\langle \widetilde{E_2}\widetilde{\beta}_2,\widetilde{\beta}_2\rangle_{\mathcal{L}^2_{\mathcal{F}}([s,T],\mathbb{R}^m)}+\langle\widetilde{G}\widetilde{\gamma},\widetilde{\gamma}\rangle_{\mathcal{L}^2_{\mathcal{F}_{T}}(\Omega,\mathbb{R}^n)}= a_0,\\
&\mbox{the following inquality holds:}\\
&\!\!\!\!\!\!\!\!\!\!\!\!\langle \widetilde{\alpha}_i,\widetilde{D}X^*\rangle_{\mathcal{L}^2_{\mathcal{F}}([s,T],\mathbb{R}^n)}+\langle \widetilde{\beta}_{1,i},\widetilde{E_1}u^*_1\rangle_{\mathcal{L}^2_{\mathcal{F}}([s,T],\mathbb{R}^m)}+
\langle \widetilde{\beta}_{2,i},\widetilde{E_2}u_2\rangle_{\mathcal{L}^2_{\mathcal{F}}([s,T],\mathbb{R}^m)}+\langle \widetilde{\gamma}_i,\widetilde{G}X^*(T)\rangle_{\mathcal{L}^2_{\mathcal{F}_{T}}(\Omega,\mathbb{R}^n)}\leq a_0\Bigg\}.
\end{align*}
Then $U_1$ can be approximately described by $U^{p}$ when $p\in \mathbb{N}^*=\{1,2,\cdots\}$ is large enough.
\end{example}
Examples \ref{+18} indicates that the set $U_1$  can be approximately described by $U^{p}$ with enough affine constraints. Thus it would be significant and interesting to study the feedback Stackelberg equilibrium for the LQ-SSDG with affine constraints (LQ-SSDG-AC, for short). However, to the best of our knowledge, there are no papers dealing with such a problem.

The present paper is thus devoted to studying  the LQ-SSDG-AC and finding its  feedback Stackelberg equilibrium. One of the main difficulties in the LQ-SSDG-AC is that the Slater condition is hard to be verified since affine constraints  depend on the triple $(X^*,u^*_1,u_2)$, which can be overcome by introducing the corresponding dual equation. Compared with the above literature reviewed, the main contributions of this paper are threefold:
 \begin{itemize}
 \item From the viewpoint of model, both equality constraints and inequality constraints are considered in the affine constraints which depend on both the follower's response and  the leader's announced strategy, and the coefficients of the cost functionals are allowed to be indefinite.
  \item From the viewpoint of methodology, by constructing new FBSDEs, we rewrite the affine constraints as some constraints depending only on $u_2$, and so the Slater condition can be more easily verified. To the best of our knowledge, this method is new and can be applied to solve some other problems.
 \item From the viewpoint of conclusions, a new positive definite condition is proposed for ensuring the uniqueness of the solutions to the dual problem.
\end{itemize}

The rest of this paper is structured as follows. The next section formulates the LQ-SSDG-AC. In Section 3, the feedback optimal strategy of the follower is obtained under some mild conditions. In section 4, the KKT condition is obtained for the feedback optimal strategy of the leader under some mild conditions, and the feedback Stackelberg equilibrium is given for the LQ-SSDG-AC.  Before concluding this paper,  two non-degenerate  examples with indefinite coefficients are given to show the effectiveness of our main results in Section 5.
\section{Preliminaries}

In this section, we present the formulation of the LQ-SSDG-AC. Based on system \eqref{SDE1}, the admissible control sets  for players can be defined as follows.
\begin{defn}
For positive integers $l,l'$ with $1\leq l'\leq l$, $\mathcal{I}_1=\{1,2,\cdots,l'\}$ and $\mathcal{I}_2=\{l'+1,l'+2,\cdots,l\}$, let the process $\alpha\in\mathcal{L}^2_{\mathcal{F}}([s,T],\mathbb{R}^{n\times l})$, processes $\beta,\gamma\in\mathcal{L}^2_{\mathcal{F}}([s,T],\mathbb{R}^{m\times l})$, and the random variable $\delta\in {\mathcal{L}^{2}_{\mathcal{F}_{T}}(\Omega,\mathbb{R}^{n\times l})}$ be given, then the follower's admissible control set  is defined by
$$U_1^{ad}=\mathcal{L}^2_{\mathcal{F}}([s,T],\mathbb{R}^m)$$
 and the leader's  is defined by
\begin{align}
U_2^{ad}=&\Bigg\{u_2\in\mathcal{L}^2_{\mathcal{F}}([s,T],\mathbb{R}^m)\big|u_2\;\mbox{satisfies the following affine constraints:}\nonumber\\
&\;\langle X^*,\alpha_i\rangle_{\mathcal{L}^2_{\mathcal{F}}([s,T],\mathbb{R}^n)}+\langle u_1^*,\beta_i\rangle_{\mathcal{L}^2_{\mathcal{F}}([s,T],\mathbb{R}^m)}+\langle u_2,\gamma_i\rangle_{\mathcal{L}^2_{\mathcal{F}}([s,T],\mathbb{R}^m)}+\langle X^*(T),\delta_i\rangle_{\mathcal{L}^2_{\mathcal{F}_{T}}(\Omega,\mathbb{R}^n)}\leq a_i\;(i\in\mathcal{I}_1),\nonumber\\
&\;\langle X^*,\alpha_i\rangle_{\mathcal{L}^2_{\mathcal{F}}([s,T],\mathbb{R}^n)}+\langle u_1^*,\beta_i\rangle_{\mathcal{L}^2_{\mathcal{F}}([s,T],\mathbb{R}^m)}+\langle u_2,\gamma_i\rangle_{\mathcal{L}^2_{\mathcal{F}}([s,T],\mathbb{R}^m)}+\langle X^*(T),\delta_i\rangle_{\mathcal{L}^2_{\mathcal{F}_{T}}(\Omega,\mathbb{R}^n)}= a_i \;(i\in\mathcal{I}_2)\Bigg\},\label{3}
\end{align}
where $\theta=(\theta_1,\theta_2,\cdots,\theta_l)$ $(\theta\in\{\alpha,\beta,\gamma,\delta\})$, $a=(a_1,a_2,\cdots,a_l)^{\top}\in \mathbb{R}^l$, and $u_1^*$ is the best response of the follower.
\end{defn}

For $(X,u_1,u_2)$ governed by the state equation \eqref{SDE1}, consider the following cost functionals
\begin{equation}\label{8}
\begin{cases}
J_1(X,u_1,u_2)=\mathbb{E}\left[{\int_s^T}\left\langle D_1X,X\right\rangle+\left\langle E_1u_1,u_1\right\rangle
dt+\left\langle G_1X(T),X(T)\right\rangle\right],\\
J_2(X,u_1,u_2)=\mathbb{E}\left[{\int_s^T}\left\langle D_2X,X\right\rangle+\left\langle E_2u_2,u_2\right\rangle
dt+\left\langle G_2X(T),X(T)\right\rangle\right]
\end{cases}
\end{equation}
for the follower and the leader, respectively, where
$D_i:[0,T]\times\Omega\rightarrow\mathcal{S}^n$, $E_i:[0,T]\times\Omega\rightarrow\mathcal{S}^m$ and $G_i:[0,T]\times\Omega\rightarrow\mathcal{S}^n$ are all bounded and $\mathcal{F}$-progressively measurable processes. Moreover, $E_i$ is invertible. We emphasize that $D_i$, $E_i$ and $G_i$ are not necessarily positive definite.

In the present paper, we would like to study the feedback Stackelberg equilibrium for the following non-zero-sum  LQ-SSDG-AC.
\begin{prob}\label{2}
For the strategy $u_2$ announced by the leader, the follower aims to find the optimal strategy $\widehat{u}_1\in U_1^{ad}$ such that
$$
J_1(X^*,u_1^*,u_2)=\inf\limits_{u_1 \in U_1^{ad}}J_1(X,u_1,u_2),
$$
where $(X^*,u_1^*)=(X^*(u_2),u_1^*(u_2))$,  and $J_1(X,u_1,u_2)$  is given in \eqref{8}. Then based on the strategy $u_1^*$ of the follower, the leader aims to find the optimal strategy $\widehat{u}_2\in  U_2^{ad}$ such that
\begin{equation}\label{+8}
J_2(\widehat{X},\widehat{u}_1,\widehat{u}_2)=\inf\limits_{u_2 \in U_2^{ad}}J_2(X^*,u_1^*,u_2),
\end{equation}
where $(\widehat{X},\widehat{u}_1)=(X^*(\widehat{u}_2),u_1^*(\widehat{u}_2))$, and $J_2(X,u_1,u_2)$  is given in \eqref{8}.
\end{prob}
\begin{remark}
$(\widehat{u}_1,\widehat{u}_2)\in U_1^{ad}\times U_2^{ad}$ is called the Stackelberg equilibrium of Problem \ref{2}. In addition, if both $\widehat{u}_1$ and $\widehat{u}_2$ can be represented as functions of the state variable, then $(\widehat{u}_1,\widehat{u}_2)$ with such an expression is called the feedback Stackelberg equilibrium of Problem \ref{2}.
\end{remark}

\section{LQ-SSDG for the Follower}
In this section, we  search for the feedback expression of the follower's optimal strategy. To this end, we need the following condition.

\noindent $(\mathbf{H1})$ There exists a constant $\epsilon_1>0$ such that for any $\widetilde{u}_1\in\mathcal{L}^2_{\mathcal{F}}([s,T],\mathbb{R}^m)$,
$$
\left\langle D_1\widetilde{X},\widetilde{X}\right\rangle_{\mathcal{L}^2_{\mathcal{F}}([s,T],\mathbb{R}^n)}+\left\langle E_1\widetilde{u}_1, \widetilde{u}_1\right\rangle_{\mathcal{L}^2_{\mathcal{F}}([s,T],\mathbb{R}^m)}+\left\langle G_1\widetilde{X}(T),\widetilde{X}(T)\right\rangle_{\mathcal{L}^2_{\mathcal{F}_{T}}(\Omega,\mathbb{R}^n)}\\
\geq\epsilon_1\| \widetilde{u}_1\|^2_{\mathcal{L}^2_{\mathcal{F}}([s,T],\mathbb{R}^m)},
$$
where $\widetilde{X}\in \mathcal{L}^2_{\mathcal{F}}(\Omega,C([s,T],\mathbb{R}^n))$ is the unique solution to the SDE
\begin{equation}\label{7}
\begin{cases}
d\widetilde{X}=\left[A\widetilde{X}+B_1\widetilde{u}_1\right]dt+C\widetilde{X}dB(t),\\
\widetilde{X}(0)=0^{n\times1}.
\end{cases}
\end{equation}
Then, the follow's optimal strategy can be given as follows.
\begin{thm}
Under $(\mathbf{H1})$, $u_1^*$ is an optimal strategy of the follower if $u_1^*=E_1^{-1}B_1^{\top}p_1^*$, where $(X^*,p_1^*,q_1^*)\in\mathcal{L}^2_{\mathcal{F}}(\Omega,C([s,T],\mathbb{R}^n))\times\mathcal{L}^2_{\mathcal{F}}(\Omega,C([s,T],\mathbb{R}^n))\times\mathcal{L}^2_{\mathcal{F}}([s,T],\mathbb{R}^n)$ is the unique solution to the FBSDE
\begin{equation}\label{+9}
\begin{cases}
dX^*=\left[AX^*+B_1u_1^*+B_2u_2\right]dt+CX^*dB(t), \quad t\in[s,T],\\
dp_1^*=\left[D_1X^*-A^{\top}p_1^*-C^{\top}q_1^*\right]dt+q_1^*dB(t), \quad t\in[s,T],\\
X^*(s)=\xi\in \mathbb{R}^n,\quad p_1^*(T)=-G_1X^*(T).
\end{cases}
\end{equation}
\end{thm}

\begin{proof}
For given $u_2\in U_2^{ad}$ and any pair $(X,u_1)$ governed by \eqref{SDE1}, let $(\widetilde{X},\widetilde{u}_1)=(X^*-X,u^*_1-u_1)$, which satisfies \eqref{7}. Then, we have
\begin{align*}
\triangle J_1=&\frac{1}{2}\left[J_1(X^*,u_1^*,u_2)-J_1(X,u_1,u_2)\right]\\
=&-\frac{1}{2}\left\langle D_1\widetilde{X},\widetilde{X}\right\rangle_{\mathcal{L}^2_{\mathcal{F}}([s,T],\mathbb{R}^n)}-\frac{1}{2}\left\langle E_1\widetilde{u}_1,\widetilde{u}_1\right\rangle_{\mathcal{L}^2_{\mathcal{F}}([s,T],\mathbb{R}^m)}
-\frac{1}{2}\left\langle G_1\widetilde{X}(T),\widetilde{X}(T)\right\rangle_{\mathcal{L}^2_{\mathcal{F}_{T}}(\Omega,\mathbb{R}^n)}\\
&\mbox{}+\left\langle D_1\widetilde{X},X^*\right\rangle_{\mathcal{L}^2_{\mathcal{F}}([s,T],\mathbb{R}^n)}+\left\langle E_1\widetilde{u}_1,u^*_1\right\rangle_{\mathcal{L}^2_{\mathcal{F}}([s,T],\mathbb{R}^m)}
+\left\langle G_1\widetilde{X}(T),X^*(T)\right\rangle_{\mathcal{L}^2_{\mathcal{F}_{T}}(\Omega,\mathbb{R}^n)}.
\end{align*}
Next, applying the It\^{o} formula to $\langle\widetilde{X},p^*_1\rangle$ derives
$$-\left\langle \widetilde{X}(T),G_1 X^*(T)\right\rangle_{\mathcal{L}^2_{\mathcal{F}_{T}}(\Omega,\mathbb{R}^n)}=\left\langle \widetilde{X},D_1 X^*\right\rangle_{\mathcal{L}^2_{\mathcal{F}}([s,T],\mathbb{R}^n)}+\left\langle B_1\widetilde{u}_1,p^*_1\right\rangle_{\mathcal{L}^2_{\mathcal{F}}([s,T],\mathbb{R}^n)}.$$
Therefore, combining the above arguments, one has
$$
\triangle J_1\leq -\frac{\epsilon_1}{2}\| \widetilde{u}_1\|^2_{\mathcal{L}^2_{\mathcal{F}}([s,T],\mathbb{R}^m)}+\left\langle \widetilde{u}_1, E_1u^*_1-B_1^{\top}p^*_1\right\rangle_{\mathcal{L}^2_{\mathcal{F}}([s,T],\mathbb{R}^m)}
=-\frac{\epsilon_1}{2}\| \widetilde{u}_1\|^2_{\mathcal{L}^2_{\mathcal{F}}([s,T],\mathbb{R}^m)},
$$
which indicates the conclusion.
\end{proof}
We further search for the feedback expression for $u_1^*$. Setting $p_1^*=-\phi_1X^*-\psi_1$, then one can check that $\phi_1$ is determined by the stochastic Riccati equations (SREs, for short)
\begin{equation}\label{9}
\begin{cases}
d\phi_1=-\left[D_1+\phi_1A+A^{\top}\phi_1+\widetilde{\phi}_1C+C^{\top}\widetilde{\phi}_1+C^{\top}\phi_1C-\phi_1S_1\phi_1\right]dt\mbox{}+\widetilde{\phi}_1dB(t), \quad t\in[s,T],\\
\phi_1(T)=G_1,
\end{cases}
\end{equation}
and $\psi_1$ is determined by the BSDE
\begin{equation}\label{4}
\begin{cases}
d\psi_1=-\left[(A^{\top}-\phi_1S_1)\psi_1+C^{\top}\widetilde{\psi}_1+\phi_1B_2u_2\right]dt+\widetilde{\psi}_1dB(t), \quad t\in[s,T],\\
\psi_1(T)=0^{n\times1},
\end{cases}
\end{equation}
where $S_1=B_1E_1^{-1}B^{\top}_1$. We have the following results.
\begin{thm}\label{42}
Under $(\mathbf{H1})$, an optimal feedback strategy of the follower can be given as:
\begin{equation}\label{10}
u_1^*=-E_1^{-1}B_1^{\top}(\phi_1X^*+\psi_1).
\end{equation}
\end{thm}
\begin{proof}
According to \cite[Theorem 6.3]{sun2021indefinite}, there exists a unique solution $(\phi_1,\widetilde{\phi}_1)\in\mathcal{L}^{\infty}_{\mathcal{F}}(\Omega,C([s,T],\mathcal{S}^n))\times\mathcal{L}^2_{\mathcal{F}}([s,T],\mathcal{S}^n)$ to \eqref{9}. Thus by \cite[Proposition 3.2]{yong2020stochastic}, for any $u_2\in U_2^{ad}$, there exists a unique solution $(\psi_1,\widetilde{\psi}_1)\in\mathcal{L}^2_{\mathcal{F}}(\Omega,C([s,T],\mathbb{R}^n))\times\mathcal{L}^2_{\mathcal{F}}([s,T],\mathbb{R}^n)$ to \eqref{4} and a unique solution
$X^*\in\mathcal{L}^2_{\mathcal{F}}(\Omega,C([s,T],\mathbb{R}^n))$ to
\begin{equation*}
\begin{cases}
dX^*=\left[(A-E_1^{-1}B_1^{\top}\phi_1)X^*-E_1^{-1}B_1^{\top}\psi_1+B_2u_2\right]dt+CX^*dB(t), \quad t\in[s,T],\\
X^*(s)=\xi\in \mathbb{R}^n.
\end{cases}
\end{equation*}
So the solvability of \eqref{+9} is guaranteed. Applying the It\^{o} formula to $\langle\phi_1X+2\psi_1,X\rangle$, one has
\begin{align*}
&\mathbb{E}\left[\langle G_1X(T),X(T)\rangle-\langle\phi_1(s)\xi+2\psi_1(s),\xi\rangle\right],\\
=&-\left\langle D_1{X},{X}\right\rangle_{\mathcal{L}^2_{\mathcal{F}}([s,T],\mathbb{R}^n)}+\left\langle S_1\phi_1{X},\phi_1{X}\right\rangle_{\mathcal{L}^2_{\mathcal{F}}([s,T],\mathbb{R}^n)}
+2\left\langle \phi_1B_1u_1,X\right\rangle_{\mathcal{L}^2_{\mathcal{F}}([s,T],\mathbb{R}^n)}\\
&\mbox{}+2\left\langle B_1u_1+S_1\phi_1{X},\psi_1\right\rangle_{\mathcal{L}^2_{\mathcal{F}}([s,T],\mathbb{R}^n)}
+2\left\langle B_2u_2,\psi_1\right\rangle_{\mathcal{L}^2_{\mathcal{F}}([s,T],\mathbb{R}^n)}.
\end{align*}
Then combining the definition of $J_1(X,u_1,u_2)$, we have
\begin{align*}
J_1(X,u_1,u_2)
=&\left\langle\mathbb{E}(\phi_1(s)\xi+2\psi_1(s)),\xi\right\rangle+2\left\langle B_2u_2,\psi_1\right\rangle_{\mathcal{L}^2_{\mathcal{F}}([s,T],\mathbb{R}^n)}-
\left\langle E_1^{-1}B_1^{\top}\psi_1,B_1^{\top}\psi_1\right\rangle_{\mathcal{L}^2_{\mathcal{F}}([s,T],\mathbb{R}^m)}\\
&\mbox{}+\left\langle E_1^{-1}(E_1u_1+B_1^{\top}\phi_1X+B_1^{\top}\psi_1),E_1u_1+B_1^{\top}\phi_1X+B_1^{\top}\psi_1\right\rangle_{\mathcal{L}^2_{\mathcal{F}}([s,T],\mathbb{R}^n)},
\end{align*}
which reaches its minimum when $E_1u_1+B_1^{\top}\phi_1X+B_1^{\top}\psi_1=0$. This ends the proof.
\end{proof}

\section{LQ-SSDG for the Leader}
In this section, we  search for  the feedback expression of the leader's optimal strategy and give the formulation of the feedback Stackelberg equilibrium. Recall \eqref{+9} and \eqref{4},  $(X^*,\psi_1,\widetilde{\psi}_1)$
satisfies the  FBSDE
\begin{equation}\label{+3}
\begin{cases}
dX^*=\left[(A-S_1\phi_1)X^*-S_1\psi_1+B_2u_2\right]dt+CX^*dB(t), \quad t\in[s,T],\\
d\psi_1=-\left[(A^{\top}-\phi_1S_1)\psi_1+C^{\top}\widetilde{\psi}_1+\phi_1B_2u_2\right]dt+\widetilde{\psi}_1dB(t), \quad t\in[s,T],\\
X^*(s)=\xi\in \mathbb{R}^n,\quad \psi_1(T)=0^{n\times1}.
\end{cases}
\end{equation}
We need the following condition for solving the leader's problem.

\noindent $(\mathbf{H2})$ There exists a constant $\epsilon_2>0$ such that for any $\widetilde{u}_2\in\mathcal{L}^2_{\mathcal{F}}([s,T],\mathbb{R}^m)$,
$$
\left\langle D_2\widetilde{X}^*,\widetilde{X}^*\right\rangle_{\mathcal{L}^2_{\mathcal{F}}([s,T],\mathbb{R}^n)}+\left\langle E_2\widetilde{u}_2, \widetilde{u}_2\right\rangle_{\mathcal{L}^2_{\mathcal{F}}([s,T],\mathbb{R}^m)}+\left\langle G_2\widetilde{X}^*(T),\widetilde{X}^*(T)\right\rangle_{\mathcal{L}^2_{\mathcal{F}_{T}}(\Omega,\mathbb{R}^n)}\\
\geq\epsilon_2\| \widetilde{u}_2\|_{\mathcal{L}^2_{\mathcal{F}}([s,T],\mathbb{R}^m)},
$$
where $(\widetilde{X}^*,\overline{\psi}_1,\widehat{\psi}_1)\in \mathcal{L}^2_{\mathcal{F}}(\Omega,C([s,T],\mathbb{R}^n))\times\mathcal{L}^2_{\mathcal{F}}(\Omega,C([s,T],\mathbb{R}^n))\times\mathcal{L}^2_{\mathcal{F}}([s,T],\mathbb{R}^n)$  is the unique solution to the FBSDE
\begin{equation}\label{11}
\begin{cases}
d\widetilde{X}^*=\left[(A-S_1\phi_1)\widetilde{X}^*-S_1\overline{\psi}_1+B_2\widetilde{u}_2\right]dt+C\widetilde{X}^*dB(t),\\
d\overline{\psi}_1=-\left[(A-S_1\phi_1)^{\top}\overline{\psi}_1+C^{\top}\widehat{\psi}_1+\phi_1B_2\widetilde{u}_2\right]+\widehat{\psi}_1dB(t),\\
\widetilde{X}^*(0)=0, \quad\overline{\psi}_1(T)=0^{n\times1}.
\end{cases}
\end{equation}

Concerning on the constraints in \eqref{3}, let
$$
\rho_i(u_2)=\langle X^*,\alpha_i\rangle_{\mathcal{L}^2_{\mathcal{F}}([s,T],\mathbb{R}^n)}+\langle u_1^*,\beta_i\rangle_{\mathcal{L}^2_{\mathcal{F}}([s,T],\mathbb{R}^m)}+\langle u_2,\gamma_i\rangle_{\mathcal{L}^2_{\mathcal{F}}([s,T],\mathbb{R}^m)}+\langle X^*(T),\delta_i\rangle_{\mathcal{L}^2_{\mathcal{F}_{T}}(\Omega,\mathbb{R}^n)},
$$
and $\rho=(\rho_1,\rho_2,\cdots,\rho_l)$, then it follows from \eqref{10} that
\begin{equation}\label{+1}
\rho(u_2)=\mathbb{E}\left(\displaystyle{\int_s^T}\left[(\alpha-\phi_1B_1E_1^{-1}\beta)^{\top}X^{*}-\beta^{\top}E_1^{-1}B_1^{\top}\psi_1+\gamma^{\top}u_2\right]dt+\delta^{\top}X^{*}(T)\right).
\end{equation}

Now we introduce the following relaxed problem for the leader's problem stated by \eqref{+8}.
\begin{prob}\label{13}
For any given $\lambda=(\lambda_1,\lambda_2,\cdots,\lambda_l)^{\top}$ with $\lambda_i\geq0$ ($i\in\mathcal{I}_1$) and $\lambda_i\in \mathbb{R}$ ($i\in\mathcal{I}_2$), find $u^{*}=u^*(\lambda)$ such that
\begin{equation}\label{+10}
\widetilde{J}(\lambda,u_2^*(\lambda))=\inf\limits_{u_2 \in \mathcal{L}^2_{\mathcal{F}}([s,T],\mathbb{R}^{n})}\widetilde{J}(\lambda,u_2),
\end{equation}
where
$$\widetilde{J}(\lambda,u_2)=J_2(X^*,u^*_1(u_2),u_2)+2\langle \rho(u_2)-a,\lambda\rangle.$$
\end{prob}
The solution to Problem \eqref{13} can be given as follows.
\begin{thm}\label{+7}
Under $(\mathbf{H1})$-$(\mathbf{H2})$, $u_2^*=u_2^*(\lambda)$ is a solution to Problem \eqref{13} if
\begin{equation}\label{5}
u^*_2=E_2^{-1}\left[B_2^{\top}p^*_2-\gamma\lambda+B_2^{\top}\phi_1Y^*\right],
\end{equation} where $(X^{**},Y^*,p_2^*,q_2^*,\psi_1^*,\widetilde{\psi}^*_1)$ is the unique solution to the fully
coupled FBSDEs
\begin{equation}\label{+4}
\begin{cases}
dX^{**}=\left[(A-S_1\phi_1)X^{**}-S_1\psi^*_1+B_2u^*_2\right]dt+CX^{**}dB_t, \quad t\in[s,T]\\
dY^*=\left[(A-S_1\phi_1)Y^*-S_1p_2^*+B_1E_1^{-1}\beta\lambda\right]dt+CY^*dB(t), \quad t\in[s,T],\\
dp_2^*=\left[D_2X^{**}-(A^{\top}-\phi_1 S_1)p_2^*-C^{\top}q_2^*+(\alpha-\phi_1B_1E_1^{-1}\beta)\lambda\right]dt+q_2^*dB(t), \quad t\in[s,T],\\
d\psi^*_1=-\left[(A^{\top}-\phi_1S_1)\psi^*_1+C^{\top}\widetilde{\psi}^*_1+\phi_1B_2u^*_2\right]dt+\widetilde{\psi}^*_1dB(t), \quad t\in[s,T],\\
X^{**}(0)=\xi,\quad Y^*(0)=0^{n\times1},\quad p_2^*(T)=-G_2X^{**}(T)-\delta\lambda,\quad\psi^*_1(T)=0^{n\times1}
\end{cases}
\end{equation}
such that $X^{**},Y^*,p_2^*,\psi_1^*\in\mathcal{L}^2_{\mathcal{F}}(\Omega,C([s,T],\mathbb{R}^n))$ and $q_2^*,\widetilde{\psi}^*_1\in\mathcal{L}^2_{\mathcal{F}}([s,T],\mathbb{R}^n)$.
\end{thm}
\begin{proof}
For any pair $(X^*,\psi_1,\widetilde{\psi}_1,u_2)$ governed by \eqref{+3}, let $$(\widetilde{X}^*,\overline{\psi}_1,\widehat{\psi}_1,\widetilde{u}_2)=(X^{**}-X^*,\psi_1^*-\psi_1,\widetilde{\psi}_1^*-\widetilde{\psi}_1,u^*_2(\lambda)-u_2),$$ which satisfies \eqref{11}. We have
\begin{align*}
\triangle J_2=&\frac{1}{2}\left[\widetilde{J}(\lambda,u_2^*(\lambda))-\widetilde{J}(\lambda,u_2)\right]\\
=&-\frac{1}{2}\left\langle D_2\widetilde{X}^*,\widetilde{X}^*\right\rangle_{\mathcal{L}^2_{\mathcal{F}}([s,T],\mathbb{R}^n)}-\frac{1}{2}\left\langle E_2\widetilde{u}_2,\widetilde{u}_2\right\rangle_{\mathcal{L}^2_{\mathcal{F}}([s,T],\mathbb{R}^m)}
-\frac{1}{2}\left\langle G_2\widetilde{X}^*(T),\widetilde{X}^*(T)\right\rangle_{\mathcal{L}^2_{\mathcal{F}_{T}}(\Omega,\mathbb{R}^n)}\\
&\mbox{}+\left\langle D_2\widetilde{X}^*,X^{**}\right\rangle_{\mathcal{L}^2_{\mathcal{F}}([s,T],\mathbb{R}^n)}+\left\langle E_2\widetilde{u}_2,u^*_2\right\rangle_{\mathcal{L}^2_{\mathcal{F}}([s,T],\mathbb{R}^m)}
+\left\langle G_2\widetilde{X}^*(T),X^{**}(T)\right\rangle_{\mathcal{L}^2_{\mathcal{F}_{T}}(\Omega,\mathbb{R}^n)}
+\langle \rho(\widetilde{u}_2),\lambda\rangle.
\end{align*}
Then, applying the It\^{o} formula to $\langle\widetilde{X}^*,p_2^*\rangle-\langle\overline{\psi}_1,Y^*\rangle$ gives
\begin{align*}
\mathbb{E}\left[\langle\widetilde{X}^*(T),-G_2X^{**}(T)\rangle\right]=&\left\langle \widetilde{X}^*,D_2X^{**}\right\rangle_{\mathcal{L}^2_{\mathcal{F}}([s,T],\mathbb{R}^n)}+\langle \rho(\widetilde{u}_2),\lambda\rangle\\
&\mbox{}+\left\langle \widetilde{u}_2,B^{\top}_2p_2^*+B^{\top}_2\phi_1Y^*-\gamma\lambda\right\rangle_{\mathcal{L}^2_{\mathcal{F}}([s,T],\mathbb{R}^m)}.
\end{align*}
Therefore, combining the above arguments, one has
$$
\triangle J_2\leq -\frac{\epsilon_2}{2}\| \widetilde{u}_2\|^2_{\mathcal{L}^2_{\mathcal{F}}([s,T],\mathbb{R}^m)}+\left\langle \widetilde{u}_2,E_2u^*_2-(B^{\top}_2p_2^*+B^{\top}_2\phi_1Y^*-\gamma\lambda)\right\rangle_{\mathcal{L}^2_{\mathcal{F}}([s,T],\mathbb{R}^m)}
=-\frac{\epsilon_2}{2}\| \widetilde{u}_2\|^2_{\mathcal{L}^2_{\mathcal{F}}([s,T],\mathbb{R}^m)},
$$
which ends the proof.
\end{proof}
Now we focus on FBSDEs \eqref{+4}. Note that $(X^{**},p_2^*,q_2^*)$ and $(Y^*,\psi_1^*,\widetilde{\psi}^*_1)$ form two fully coupled FBSDEs. Thus by setting $S_2=B_2E_2^{-1}B^{\top}_2$ and
$$
Z=
\begin{bmatrix}
X^{**}\\
Y^*
\end{bmatrix},\quad
P=
\begin{bmatrix}
p_2^*\\
\psi_1^*
\end{bmatrix},\quad
Q=
\begin{bmatrix}
q_2^*\\
\widetilde{\psi}_1^*
\end{bmatrix},\quad
\overline{F}=
\begin{bmatrix}
S_2 &-S_1\\
-S_1&0^{n\times n}
\end{bmatrix},\quad\overline{C}=
\begin{bmatrix}
C& 0^{n\times n}\\
0^{n\times n}& C
\end{bmatrix},
$$
$$
\overline{A}=
\begin{bmatrix}
A-S_1\phi_1&S_2\phi_1\\
0^{n\times n}& A-S_1\phi_1
\end{bmatrix},\quad
\overline{B}_2=
\begin{bmatrix}
0^{n\times m}\\
\phi_1B_2
\end{bmatrix},\quad
\overline{D}=
\begin{bmatrix}
D_2&0^{n\times n}\\
0^{n\times n}&\phi_1S_2\phi_1
\end{bmatrix},
$$
the FBSDEs \eqref{+4} can be rewritten as
\begin{equation}\label{12}
\begin{cases}
dZ=\left[\overline{A}Z+\overline{F}P+
\begin{bmatrix}
-B_2E_2^{-1}\gamma\\
B_1E_1^{-1}\beta
\end{bmatrix}\lambda\right]dt+\overline{C}ZdB_t, \quad t\in[s,T]\\
dP=\left[\overline{D}Z-\overline{A}^{\top}P-\overline{C}^{\top}Q+
\begin{bmatrix}
\alpha-\phi_1B_1E_1^{-1}\beta\\
\phi_1B_2E_2^{-1}\gamma
\end{bmatrix}\lambda\right]dt+QdB(t), \quad t\in[s,T],\\
Z(s)=
\begin{bmatrix}
\xi\\
0^{n\times1}
\end{bmatrix},\quad P(T)=\begin{bmatrix}
-G_2X^{**}(T)-\delta\lambda\\
0^{n\times1}
\end{bmatrix},
\end{cases}
\end{equation}
and the optimal control \eqref{5} can be rewritten as
\begin{equation}\label{1}
u_2^*=E_2^{-1}\left(\begin{bmatrix}
B_2^{\top}&
0^{m\times n}
\end{bmatrix}P+\begin{bmatrix}
0^{m\times n}&
B_2^{\top}\phi_1
\end{bmatrix}Z-\gamma\lambda\right).
\end{equation}
Following \cite{ma1999forward, yong2002leader}, we now use the idea of the four-step scheme to study the solvability of the new FBSDE \eqref{12}. Setting $P=-\phi_2Z-\psi_2\lambda$, one can check that $(\phi_2,\widetilde{\phi}_2)$ is governed by the SRE
\begin{equation}\label{+5}
\begin{cases}
d\phi_2=-\left[\overline{A}^{\top}\phi_2+\phi_2\overline{A}+\overline{C}^{\top}\widetilde{\phi}_2+\widetilde{\phi}_2\overline{C}+\overline{C}^{\top}{\phi}_2\overline{C}
-\phi_2\overline{F}\phi_2+\overline{D}\right]dt+\widetilde{\phi}_2dB(t), \quad t\in[s,T],\\
\phi_2(T)=\begin{bmatrix}
G_2&0^{n\times n}\\
0^{n\times n}&0^{n\times n}
\end{bmatrix},
\end{cases}
\end{equation}
and $(\psi_2,\widetilde{\psi}_2)$ is governed by the BSDE
\begin{equation}\label{+6}
\begin{cases}
d\psi_2=-\Bigg[(\overline{A}^{\top}-\phi_2\overline{F})\psi_2+\overline{C}^{\top}\widetilde{\psi}_2
+\begin{bmatrix}
\alpha-\phi_1B_1E_1^{-1}\beta\\
\phi_1B_2E_2^{-1}\gamma
\end{bmatrix}+\phi_2\begin{bmatrix}
-B_2E_2^{-1}\gamma\\
B_1E_1^{-1}\beta
\end{bmatrix}\Bigg]dt+\widetilde{\psi}_2dB(t), \quad t\in[s,T],\\
\psi_2(T)=\begin{bmatrix}
\delta\\
0^{n\times l}
\end{bmatrix}.
\end{cases}
\end{equation}

It is worth mentioning that
the SRE \eqref{+5} is not standard because $\overline{F}$ is not positive definite in general. Thus we need the following assumption for ensuring the solvability of \eqref{+5}.

\noindent $(\mathbf{H3})$ There exists a unique solution $(\phi_2,\widetilde{\phi}_2)\in\mathcal{L}^{\infty}_{\mathcal{F}}(\Omega,C([s,T],\mathcal{S}^{2n}))\times\mathcal{L}^2_{\mathcal{F}}([s,T],\mathcal{S}^{2n})$ to \eqref{+5}.

Then, depending on the new state variable $Z$, the feedback solution $u_2^*=u_2^*(\lambda)$ to Problem \ref{13} can be given as follows.
\begin{thm}\label{38}
Under  $(\mathbf{H1})$-$(\mathbf{H3})$, one feedback solution to Problem \eqref{13} can be given as:
\begin{equation}\label{15}
u_2^*=\left(\begin{bmatrix}
0^{m\times n}&
E_2^{-1}B_2^{\top}\phi_1
\end{bmatrix}-\begin{bmatrix}
E_2^{-1}B_2^{\top}&
0^{m\times n}
\end{bmatrix}\phi_2\right)Z-\left(\begin{bmatrix}
E_2^{-1}B_2^{\top}&
0^{m\times n}
\end{bmatrix}\psi_2+E_2^{-1}\gamma\right)\lambda.
\end{equation}
\end{thm}
\begin{proof}
Under condition $(\mathbf{H3})$, since $\mathcal{L}^2_{\mathcal{F}}([s,T],\mathbb{R}^{2n\times l})\subset\mathcal{L}^2_{\mathcal{F}}(\Omega,\mathcal{L}^1([s,T],\mathbb{R}^{2n\times l}))$,  there exists a unique solution $(\psi_2,\widetilde{\psi}_2)\in\mathcal{L}^2_{\mathcal{F}}(\Omega,C([s,T],\mathbb{R}^{2n\times l}))\times\mathcal{L}^2_{\mathcal{F}}([s,T],\mathbb{R}^{2n\times l})$ to the BSDE \eqref{+6} (\cite[Proposition 3.2]{yong2020stochastic}). Moreover, substituting $P=-\phi_2Z-\psi_2\lambda$ into \eqref{12}, it is easy to observe that \eqref{12} admits a unique solution $(Z,P,Q)\in\mathcal{L}^2_{\mathcal{F}}(\Omega,C([s,T],\mathbb{R}^{2n}))\times\mathcal{L}^2_{\mathcal{F}}(\Omega,C([s,T],\mathbb{R}^{2n}))\times\mathcal{L}^2_{\mathcal{F}}([s,T],\mathbb{R}^{2n})$ (\cite[Theorem 1.25]{yong2020stochastic}), and so the solvability condition of \eqref{+4} in Theorem \ref{+7} can be satisfied. Thus by Theorem \ref{+7},  $u_2^*$ given by \eqref{15} is the optimal feedback strategy of the leader.
\end{proof}
Before introducing the dual problem for the leader, we need to derive  much simpler expressions of $\rho(u_2^*(\lambda))$ and the value function $\widetilde{J}(\lambda,u_2^*(\lambda))$, which can be presented by
the following two propositions.
\begin{prop}\label{+2}
It holds that
$$
\rho(u_2^*(\lambda))=\mathbb{E}[\psi_2^{\top}(s)]Z(s)+\mathbb{E}\left(\displaystyle{\int_s^T}\left[\psi_2^{\top}\begin{bmatrix}
-B_2E_2^{-1}\gamma\\
B_1E_1^{-1}\beta
\end{bmatrix}+\begin{bmatrix}
-B_2E_2^{-1}\gamma\\
B_1E_1^{-1}\beta
\end{bmatrix}^{\top}\psi_2-\gamma^{\top}E_2^{-1}\gamma-\psi_2^{\top}\overline{F}\psi_2\right]dt\right)\lambda.
$$
\end{prop}
\begin{proof}
Applying the It\^{o} formula to $\psi_2^{\top}Z$, one has
\begin{align*}
\mathbb{E}[\psi_2^{\top}(T)Z(T)-\psi_2^{\top}(s)Z(s)]=&\mbox{}-\mathbb{E}\displaystyle{\int_s^T}\left[\begin{bmatrix}
\alpha-\phi_1B_1E_1^{-1}\beta\\
\phi_1B_2E_2^{-1}\gamma
\end{bmatrix}^{\top}Z+\begin{bmatrix}
-B_2E_2^{-1}\gamma\\
B_1E_1^{-1}\beta
\end{bmatrix}^{\top}\phi_2Z\right]dt\\
&\mbox{}+\mathbb{E}\displaystyle{\int_s^T}\psi_2^{\top}\left[\overline{F}P+\begin{bmatrix}
-B_2E_2^{-1}\gamma\\
B_1E_1^{-1}\beta
\end{bmatrix}\lambda+\overline{F}\phi_2Z\right]dt\\
=&\mathbb{E}\displaystyle{\int_s^T}\left[\begin{bmatrix}
-B_2E_2^{-1}\gamma\\
B_1E_1^{-1}\beta
\end{bmatrix}^{\top}(P+\psi_2\lambda)-(\alpha-\phi_1B_1E_1^{-1}\beta)^{\top}X^{**}\right]dt\\
&\mbox{}+\mathbb{E}\displaystyle{\int_s^T}\left[-\gamma^{\top}E_2^{-1}B_2^{\top}\phi_1Y^*+\psi_2^{\top}\begin{bmatrix}
-B_2E_2^{-1}\gamma\\
B_1E_1^{-1}\beta
\end{bmatrix}\lambda-\psi_2^{\top}\overline{F}\psi_2\lambda\right]dt.
\end{align*}
Then according to \eqref{5}, we have
\begin{align*}
\mathbb{E}\left[\delta^{\top}X^{**}(T)-\psi_2^{\top}(s)Z(s)\right]=&-\mathbb{E}{\int_s^T}\left[(\alpha-\phi_1B_1E_1^{-1}\beta)^{\top}X^{**}+\gamma^{\top}u_2^*+\gamma^{\top}E_2^{-1}\gamma\lambda+\psi_2^{\top}\overline{F}\psi_2\lambda\right]dt\\
&\mbox{}+\mathbb{E}\displaystyle{\int_s^T}\left[\beta^{\top}E_1^{-1}B_1^{\top}\psi_1^*+\psi_2^{\top}\begin{bmatrix}
-B_2E_2^{-1}\gamma\\
B_1E_1^{-1}\beta
\end{bmatrix}\lambda+\begin{bmatrix}
-B_2E_2^{-1}\gamma\\
B_1E_1^{-1}\beta
\end{bmatrix}^{\top}\psi_2\lambda\right]dt.
\end{align*}
Combining \eqref{+1}, the result follows immediately.
\end{proof}

\begin{prop}\label{39}
 We have
\begin{equation}\label{18}
\widetilde{J}(\lambda,u_2^*(\lambda))=Z^{\top}(s)\mathbb{E}[\phi_2(s)]Z(s)+2\langle\mathbb{E}[\psi^{\top}_2(s)Z(s)]- a,\lambda\rangle-\langle S_0\lambda,\lambda\rangle,
\end{equation}
where
\begin{equation}\label{+12}
S_0=-\mathbb{E}\left(\displaystyle{\int_s^T}\left[\psi_2^{\top}\begin{bmatrix}
-B_2E_2^{-1}\gamma\\
B_1E_1^{-1}\beta
\end{bmatrix}+\begin{bmatrix}
-B_2E_2^{-1}\gamma\\
B_1E_1^{-1}\beta
\end{bmatrix}^{\top}\psi_2-\gamma^{\top}E_2^{-1}\gamma-\psi_2^{\top}\overline{F}\psi_2\right]dt\right).
\end{equation}
\end{prop}
\begin{proof}
Applying the It\^{o} formula to $\langle X^{**},p_2^*\rangle-\langle Y^*,\psi_2^*\rangle$ derives
\begin{align*}
&\mathbb{E}\left[\langle X^{**}(T),p_2^*(T)\rangle-\langle Y^*(T),\psi_2^*(T)\rangle-\langle X^{**}(s),p_2^*(s)\rangle+\langle Y^*(s),\psi_2^*(s)\rangle\right]\\
=&\left\langle B_2u_2^*,p^*_2\right\rangle_{\mathcal{L}^2_{\mathcal{F}}([s,T],\mathbb{R}^m)}+\left\langle D_2X^{**}+(\alpha-\phi_1B_1E_1^{-1}\beta)\lambda,X^{**}\right\rangle_{\mathcal{L}^2_{\mathcal{F}}([s,T],\mathbb{R}^n)}\\
&\mbox{}-\left\langle B_1E_1^{-1}\beta\lambda,\psi^*_1\right\rangle_{\mathcal{L}^2_{\mathcal{F}}([s,T],\mathbb{R}^n)}+\left\langle Y^*,\phi_1B_2u_2^*\right\rangle_{\mathcal{L}^2_{\mathcal{F}}([s,T],\mathbb{R}^n)}.
\end{align*}
Then, it follows from \eqref{1} that
\begin{align*}
&\mathbb{E}\left[\langle X^{**}(T),-G_2X^{**}(T)-\delta\lambda\rangle-\langle \xi,p_2^*(s)\rangle\right]\\
=&\mathbb{E}\left[\langle X^{**}(T),-G_2X^{**}(T)-\delta\lambda\rangle\right]+Z^{\top}(s)\mathbb{E}[\phi_2(s)]Z(s)+Z^{\top}(s)\mathbb{E}[\psi_2(s)]\lambda\\
=&\left\langle E_2u_2^*+\gamma\lambda,u_2^*\right\rangle_{\mathcal{L}^2_{\mathcal{F}}([s,T],\mathbb{R}^m)}+\left\langle D_2X^{**},X^{**}\right\rangle_{\mathcal{L}^2_{\mathcal{F}}([s,T],\mathbb{R}^n)}\\
&\mbox{}+\left\langle (\alpha-\phi_1B_1E_1^{-1}\beta)\lambda,X^{**}\right\rangle_{\mathcal{L}^2_{\mathcal{F}}([s,T],\mathbb{R}^n)}-\left\langle B_1E_1^{-1}\beta\lambda,\psi^*_1\right\rangle_{\mathcal{L}^2_{\mathcal{F}}([s,T],\mathbb{R}^n)}.
\end{align*}
And so  by the definition of $\widetilde{J}(\lambda,u_2)$ in Problem \ref{13}, we know that
$$
\widetilde{J}(\lambda,u^*_2(\lambda))=Z^{\top}(s)\mathbb{E}[\phi_2(s)]Z(s)+Z^{\top}(s)\mathbb{E}[\psi_2(s)]\lambda+\langle\rho(u_2^*(\lambda)),\lambda\rangle
-2\langle a,\lambda\rangle.
$$
Thus the result can be obtained from Proposition \ref{+2} immediately.
\end{proof}

We can now introduce the dual problem for the leader.
\begin{prob}\label{19}
Find $\lambda^* \in U_0$ such that
$$
\widetilde{J}(\lambda^*,u_2^*(\lambda^*))=\sup\limits_{\lambda \in U_0}\widetilde{J}(\lambda,u_2^*(\lambda)),
$$
where the Lagrangian dual function $\widetilde{J}(\lambda,u_2^*(\lambda))$ is given by \eqref{18}, and $$
U_0=\left\{\lambda\in \mathbb{R}^l\big|\lambda_i\geq0\;(i=1,2,\cdots,l')\right\}.$$
\end{prob}
\begin{remark}\label{50}
Clearly, $U_0$ is nonempty, closed and convex. If $U_0$ is in addition  bounded in $\mathbb{R}^l$, then Problem \ref{19} can be uniquely solvable because $\widetilde{J}(\lambda,u_2^*(\lambda))$ is continuous and quadratic with respect to $\lambda$. On the other hand, if  the positive definite condition $S_0\in{S}_{++}^l$ holds, then Problem \ref{19} can be also uniquely solvable.
\end{remark}
It is an important issue to investigate the well-know  strong duality between Problem \ref{19} and the problem stated by \eqref{+8} under the Slater condition. To this end, let
\begin{equation}\label{20}
\begin{cases}
d\overline{X}_i=\left[(A-S_1\phi_1)\overline{X}_i-S_1\overline{p}_i-B_1E_1^{-1}\beta_i\right]dt+C\overline{X}_idB(t), \quad t\in[s,T]\\
d\overline{p}_i=-\left[(A^{\top}-\phi_1 S_1)\overline{p}_i+C^{\top}\overline{q}_i+(\alpha_i-\phi_1B_1E_1^{-1}\beta_i)\right]dt+\overline{q}_idB(t), \quad t\in[s,T],\\
\overline{X}_i(0)=0^{n\times1},\quad \overline{p}_i(T)=-\delta_i.
\end{cases}
\end{equation}
By standard arguments, one can easily check that there exists a unique solution
$$(\overline{X}_i,\overline{p}_i,\overline{q}_i)\in \mathcal{L}^2_{\mathcal{F}}(\Omega,C([s,T],\mathbb{R}^{n}))\times\mathcal{L}^2_{\mathcal{F}}(\Omega,C([s,T],\mathbb{R}^{n}))\times\mathcal{L}^2_{\mathcal{F}}([s,T],\mathbb{R}^{n})
$$ to \eqref{20} such that
\begin{equation*}
\begin{cases}
\overline{X}_i(t)=-W^{-1}(t)\left[\int_s^tW(\varsigma)(S_1(\varsigma)\overline{p}_i(\varsigma)+B_1(\varsigma)E_1^{-1}(\varsigma)\beta_i(\varsigma))d\varsigma\right]dt, \quad t\in[s,T]\\
\overline{p}_i(t)=V^{-1}(t)\mathbb{E}\left[-V(T)\delta_i+\int_t^TV(\varsigma)(\alpha_i(\varsigma)-\phi_1(\varsigma)B_1(\varsigma)E_1^{-1}(\varsigma)\beta_i(\varsigma))\right]d\varsigma+\overline{q}_idB(t), \quad t\in[s,T],
\end{cases}
\end{equation*}
where $(W,V)\in \mathcal{L}^2_{\mathcal{F}}(\Omega,C([s,T],\mathbb{R}^{n\times n}))\times\mathcal{L}^2_{\mathcal{F}}(\Omega,C([s,T],\mathbb{R}^{n\times n}))$ is the unique invertible solution to the SDEs
\begin{equation*}
\begin{cases}
dW=W(C^2-A+S_1\phi_1)dt+WC^2ddB(t), \quad t\in[s,T]\\
dV=V(A^{\top}-\phi_1 S_1)dt+VC^{\top}dB(t), \quad t\in[s,T],\\
W(s)=V(s)=I_n.
\end{cases}
\end{equation*}
Applying the It\^{o} formula to $\langle \overline{X}_i,\psi_1\rangle-\langle \overline{q}_i,X^*\rangle$, one can check that
$$
\rho_i=\left\langle u_2,B_2^{\top}(\overline{p}_i-\phi_1\overline{X}_i)+\gamma_i\right\rangle_{\mathcal{L}^2_{\mathcal{F}}([s,T],\mathbb{R}^m)}-\left\langle\mathbb{E}[\overline{p}_i(0)],\xi\right\rangle.
$$
Thus \eqref{3} can be rewritten as
\begin{align*}
U_2^{ad}=&\Bigg\{u_2\in\mathcal{L}^2_{\mathcal{F}}([s,T],\mathbb{R}^m)\big|u_2\;\mbox{satisfies the following affine constraints:}\\
&\;\left\langle u_2,B_2^{\top}(\overline{p}_i-\phi_1\overline{X}_i)+\gamma_i\right\rangle_{\mathcal{L}^2_{\mathcal{F}}([s,T],\mathbb{R}^m)}\leq a_i+\left\langle\mathbb{E}[\overline{p}_i(0)],\xi\right\rangle,
\quad (i\in\mathcal{I}_1),\\
&\;\;\left\langle u_2,B_2^{\top}(\overline{p}_i-\phi_1\overline{X}_i)+\gamma_i\right\rangle_{\mathcal{L}^2_{\mathcal{F}}([s,T],\mathbb{R}^m)}= a_i+\left\langle\mathbb{E}[\overline{p}_i(0)],\xi\right\rangle,
\quad (i\in\mathcal{I}_2) \Bigg\}.
\end{align*}
We need the following Slater condition.

\noindent $(\mathbf{H4})$  There exists at least one $\overline{u}_2\in U^{ad}$ such that
\begin{equation}\label{24}
\begin{cases}
\left\langle \overline{u}_2,\widetilde{\rho}_i\right\rangle_{\mathcal{L}^2_{\mathcal{F}}([s,T],\mathbb{R}^m)}< \widetilde{a}_i, \;(i\in\mathcal{I}_1),\\
\left\langle \overline{u}_2,\widetilde{\rho}_i\right\rangle_{\mathcal{L}^2_{\mathcal{F}}([s,T],\mathbb{R}^m)}=\widetilde{a}_i, \;(i\in\mathcal{I}_2),
\end{cases}
\end{equation}
where
$$\widetilde{a}_i=a_i+\left\langle\mathbb{E}[\overline{p}_i(0)],\xi\right\rangle, \quad \widetilde{\rho}_i=B_2^{\top}(\overline{p}_i-\phi_1\overline{X}_i)+\gamma_i,$$
and $(\overline{X}_i,\overline{p}_i,\overline{q}_i)$ is the unique solution to \eqref{20}.

\begin{remark}\label{28}
Compared with existing literatures, since by constructing the FBSDE we rewrite the affine constrains as the ones depending only on the leader's strategy, the Slater condition \eqref{24} can be more easily verified. Furthermore, without loss of ambiguity, we require that the vector group $\{\widetilde{\rho}_i\}_{i=l'+1}^l$ of equality constraints in \eqref{24} is linearly independent, i.e., if there exists a vector $k=(k_{l'
+1},k_{l'+2},\cdots,k_l)^{\top}\in\mathbb{R}^{l-l'}$ such that $\sum \limits_{i=l'+1}^{l}k_i\widetilde{\rho}_i=0^{n\times 1}$ for all $t\in[s,T],\;\mathbb{P}$-a.s., then $k=0^{(l-l')\times 1}$.
\end{remark}

Then, the strong duality result between Problem \ref{19} and the leader's problem stated by \eqref{+8} can be specified as follows.
\begin{thm}\label{40}
Under $(\mathbf{H1})$-$(\mathbf{H4})$, Problems \ref{2} and \ref{19} are solvable. Moreover, one has
\begin{equation}\label{95}
J_2(\widehat{X},\widehat{u}_1,\widehat{u}_2)=\widetilde{J}(\lambda^*,u_2^*(\lambda^*)),
\end{equation}
where  $J_2(X,u_1,u_2)$ and $\widetilde{J}(\lambda,u_2^*(\lambda))$ are given by \eqref{8} and \eqref{18}, respectively.
\end{thm}
\begin{proof}
Consider the following two sets
\begin{align*}
U_1=\Big\{&(\chi,\mu)\in\mathbb{R}\times \mathbb{R}^l|\exists u\in\mathcal{L}^2_{\mathcal{F}}([s,T],\mathbb{R}^m),\;s.t.\;\left\langle u,\widetilde{\rho}_i\right\rangle_{\mathcal{L}^2_{\mathcal{F}}([s,T],\mathbb{R}^m)}-\widetilde{a}_i\leq\mu_i\; (i\in\mathcal{I}_1),\\
&\left\langle u,\widetilde{\rho}_i\right\rangle_{\mathcal{L}^2_{\mathcal{F}}([s,T],\mathbb{R}^m)}-\widetilde{a}_i=\mu_i (i\in\mathcal{I}_2),\;\mbox{and}\;J_2(X,u_1,u_2)-J_2(\widehat{X},\widehat{u}_1,\widehat{u}_2)\leq\chi\Big\},\\
U_2=\Big\{&(\chi,0^{l\times 1})\in\mathbb{R}\times \mathbb{R}^l,\,\chi<0\Big\}
\end{align*}
with $U_1\cap U_2=\emptyset$. Under $(\mathbf{H4})$,  both $U_1$ and $U_2$ are convex. Then, by the separation theorem of convex sets, there exists $(\overline{\chi},\overline{\mu})\in\mathbb{R}\times \mathbb{R}^l$ such that
\begin{equation}\label{3+1}
\inf\limits_{(\chi,\mu)\in U_1}[\overline{\chi}\chi+\langle \overline{\mu},\mu\rangle]\geq \sup\limits_{\chi<0}\overline{\chi}\chi,
\end{equation}
where $\overline{\mu}=(\overline{\mu}_1,\overline{\mu}_2,\cdots,\overline{\mu}_l)^{\top}$.  If $\overline{\chi}<0$,
by choosing $\chi$ large enough on the left side in \eqref{3+1}, one has
$$\inf\limits_{(\chi,\mu)\in U_1}[\overline{\chi}\chi+\langle \overline{\mu},\mu\rangle]<0, \quad
\sup\limits_{\chi<0}\overline{\chi}\chi\ge 0,
$$
which is a contradiction and so $\overline{\chi}\geq0$. Similarly, one has $\overline{\mu}_i\geq0\; (i=1,2,\cdots,l')$. Thus,
\begin{equation}\label{96}
\inf\limits_{(\chi,\mu)\in U_1}[\overline{\chi}\chi+\langle \overline{\mu},\mu\rangle]\geq0.
\end{equation}

We now claim that $\overline{\chi}>0$. In fact, if $\overline{\chi}=0$, i.e.,
\begin{equation*}
\inf\limits_{(\chi,\mu)\in U_1}[\langle \overline{\mu},\mu\rangle]=
\inf\limits_{(\chi,\mu)\in U_1}\left[\sum\limits_{i=1}^{l'}\overline{\mu}_i\mu_i+\sum\limits_{i=l'+1}^{l}\overline{\mu}_i\mu_i\right]\geq0,
\end{equation*}
then $(\mathbf{H4})$  shows that there exists  $\overline{u}\in\mathcal{L}^2_{\mathcal{F}}([s,T],\mathbb{R}^m)$ satisfying \eqref{24}. Therefore one can choose $\mu_i<0\; (i=1,2,\cdots,l')$ and $\mu_i=0\; (i=l'+1,l'+2,\cdots,l)$ in $U_1$, which leads  to $\overline{\mu}_i=0\; (i=1,2,\cdots,l')$ and
\begin{equation}\label{97}
\inf\limits_{(\chi,\mu)\in U_1}\left[\sum\limits_{i=1}^{l'}\overline{\mu}_i\mu_i+\sum\limits_{i=l'+1}^{l}\overline{\mu}_i\mu_i\right]
=\inf\limits_{(\chi,\mu)\in U_1}\sum\limits_{i=l'+1}^{l}\overline{\mu}_i\mu_i\geq0.
\end{equation}
By Remark \ref{28}, we have
$$\widetilde{\rho}=\sum\limits_{i=l'+1}^{l}\overline{\mu}_i\widetilde{\rho}_i\neq0, \quad \left\langle \overline{u},\widetilde{\rho}\right\rangle_{\mathcal{L}^2_{\mathcal{F}}([s,T],\mathbb{R}^m)}
=\sum\limits_{i=l'+1}^{l}\overline{\mu}_i\widetilde{a}_i.$$
Choosing $u_0=\overline{u}-\varepsilon\widetilde{\rho}$ for $\varepsilon>0$ small enough derives
$$
\left\langle u_0,\widetilde{\rho}\right\rangle_{\mathcal{L}^2_{\mathcal{F}}([s,T],\mathbb{R}^m)}
-\sum\limits_{i=l'+1}^{l}\overline{\mu}_i\widetilde{a}_i=
-\varepsilon\|\widetilde{\rho}\|_{\mathcal{L}^2_{\mathcal{F}}([s,T],\mathbb{R}^m)}<0,
$$
which contradicts \eqref{97}. Thus $\overline{\chi}>0$ and $\sup\limits_{\chi<0}\overline{\chi}\chi=0$.

Next, letting $\overline{\lambda}=\frac{\overline{\mu}}{\overline{\chi}}=(\overline{\lambda}_1,\overline{\lambda}_2,\cdots,\overline{\lambda}_l)^{\top}\in U_0$
and combining \eqref{96}, we obtain
\begin{align*}
0&\leq\inf\limits_{(\chi,\mu)\in U_1}[\kappa+\langle \overline{\lambda},\mu\rangle]\\
&\leq\inf_{u_2\in\mathcal{L}^2_{\mathcal{F}}([s,T],\mathbb{R}^m)}\left[J_2(X,u_1,u_2)-J_2(\widehat{X},\widehat{u}_1,\widehat{u}_2))
+\sum\limits_{i=1}^{l}\overline{\lambda}_i\left[\left\langle u_2,\widetilde{\rho}_i\right\rangle_{\mathcal{L}^2_{\mathcal{F}}([s,T],\mathbb{R}^m)}
-\widetilde{a}_i\right]\right]\\
&=\widetilde{J}(\overline{\lambda},u_2^*(\overline{\lambda}))-J_2(\widehat{X},\widehat{u}_1,\widehat{u}_2),
\end{align*}
and so
\begin{equation}\label{30}
J_2(\widehat{X},\widehat{u}_1,\widehat{u}_2)\leq\widetilde{J}(\overline{\lambda},u_2^*(\overline{\lambda}))\leq\sup\limits_{\lambda\in U_0}\widetilde{J}(\lambda,u^*_2)=\sup\limits_{\lambda\in U_0}\inf\limits_{u_2\in\mathcal{L}^2_{\mathcal{F}}([s,T],\mathbb{R}^m)}\widetilde{J}(\lambda,u_2)=\widetilde{J}(\lambda^*,u_2^*(\lambda^*)).
\end{equation}
On the other hand, since $\widehat{u}_2\in U_2^{ad}$, we have
\begin{equation}\label{31}
\sup\limits_{\lambda\in U_0}\inf\limits_{u_2\in\mathcal{L}^2_{\mathcal{F}}([s,T],\mathbb{R}^m)}\widetilde{J}(\lambda,u_2)
\leq\inf\limits_{u_2\in\mathcal{L}^2_{\mathcal{F}}([s,T],\mathbb{R}^m)}\sup\limits_{\lambda\in U_0}\widetilde{J}(\lambda,u_2)
\leq \sup\limits_{\lambda\in U_0}\widetilde{J}(\lambda,\widehat{u}_2)\leq J_2(\widehat{X},\widehat{u}_1,\widehat{u}_2).
\end{equation}
Therefore, by \eqref{30} and \eqref{31}, we see that $$J_2(\widehat{X},\widehat{u}_1,\widehat{u}_2)=\widetilde{J}(\lambda^*,u_2^*(\lambda^*))
=\widetilde{J}(\overline{\lambda},u_2^*(\overline{\lambda}))$$ and so
$\overline{\lambda}$ is a solution to Problem \ref{19}, which ends the proof.
\end{proof}
Next, according to Theorems \ref{38} and \ref{40}, Proposition \ref{39} and Remark \ref{50},   we are able to give the following KKT condition for the optimal feedback strategy of the leader by standard arguments (see \cite{bonnans2013perturbation, guo2013mathematical, jane2005necessary}  and references therein).
 \begin{thm}\label{43}
Under $(\mathbf{H1})$-$(\mathbf{H4})$, $\widehat{u}_2$ is an optimal feedback strategy of the leader if the following KKT condition holds:
\begin{equation}\label{45}
\begin{cases}
\widehat{u}_2=\left(\begin{bmatrix}
0^{m\times n}&
E_2^{-1}B_2^{\top}\phi_1
\end{bmatrix}-\begin{bmatrix}
E_2^{-1}B_2^{\top}&
0^{m\times n}
\end{bmatrix}\phi_2\right)\widehat{Z}-\left(\begin{bmatrix}
E_2^{-1}B_2^{\top}&
0^{m\times n}
\end{bmatrix}\psi_2+E_2^{-1}\gamma\right)\lambda^*,\\
\lambda^*_i\geq 0\;(i\in\mathcal{I}_1),\quad
\rho_i(\lambda^*)- a_i\leq0\;(i\in\mathcal{I}_1),\\
\lambda^*_i\left(\rho_i(\lambda^*)- a_i\right)=0\;(i\in\mathcal{I}_1),\quad
\rho_i(\lambda^*)= a_i\;(i\in\mathcal{I}_2),
\end{cases}
\end{equation}
where $\rho(u_2^*(\lambda^*))=(\rho_1(\lambda^*),\rho_2(\lambda^*),\cdots,\rho_l(\lambda^*))$ is given in Proposition \ref{+2}, $(\widehat{Z},\widehat{P},\widehat{Q})$ is the unique solution to \eqref{12}, and $(\psi_2,\widetilde{\psi}_2)$ is the unique solution to \eqref{+6}. Moreover, if $S_0\in S_{++}^l$, then $\lambda^*$ given by \eqref{45} is the unique solution to Problem \ref{19}.
\end{thm}
Finally, combining Theorems \ref{42} and \ref{43}, we obtain the following main results for Problem \ref{2}.
 \begin{thm}\label{44}
Under $(\mathbf{H1})$-$(\mathbf{H4})$, $(\widehat{u}_1,\widehat{u}_2)$ is a  feedback Stackelberg equilibrium of Problem \ref{2} if the following holds:
\begin{equation}\label{35}
\begin{cases}
\widehat{u}_1=\left(\begin{bmatrix}
0^{m\times n}&
E_1^{-1}B_1^{\top}
\end{bmatrix}\phi_2-\begin{bmatrix}
E_1^{-1}B_1^{\top}\phi_1&
0^{m\times n}
\end{bmatrix}\right)\widehat{Z}-\begin{bmatrix}
0^{m\times n}&
E_1^{-1}B_1^{\top}
\end{bmatrix}\psi_2\lambda^*,\\
\widehat{u}_2=\left(\begin{bmatrix}
0^{m\times n}&
E_2^{-1}B_2^{\top}\phi_1
\end{bmatrix}-\begin{bmatrix}
E_2^{-1}B_2^{\top}&
0^{m\times n}
\end{bmatrix}\phi_2\right)\widehat{Z}-\left(\begin{bmatrix}
E_2^{-1}B_2^{\top}&
0^{m\times n}
\end{bmatrix}\psi_2+E_2^{-1}\gamma\right)\lambda^*,\\
\lambda^*_i\geq 0\;(i\in\mathcal{I}_1),\quad
\rho_i(\lambda^*)- a_i\leq0\;(i\in\mathcal{I}_1),\\
\lambda^*_i\left(\rho_i(\lambda^*)- a_i\right)=0\;(i\in\mathcal{I}_1),\quad
\rho_i(\lambda^*)= a_i\;(i\in\mathcal{I}_2),
\end{cases}
\end{equation}
where  $(\widehat{Z},\widehat{P},\widehat{Q})$  and $(\psi_2,\widetilde{\psi}_2)$ are the unique solutions to \eqref{12} and \eqref{+6}, respectively.
\end{thm}

\begin{remark}
It is worth  mentioning that the feedback Stackelberg equilibrium $(\widehat{u}_1,\widehat{u}_2)$ given by \eqref{35} may be nonunique because that the indefiniteness of the coefficients of the cost functionals makes the Pontryagin maximum principle unapplicable. One may learn from the method of separation of variables in \cite{sun2021indefinite} to obtain the uniqueness of $(\widehat{u}_1,\widehat{u}_2)$.
\end{remark}

\section{Examples}
In this section, we provide two  examples to illustrate our main results.
Consider the controlled systems
\begin{equation}\label{60}
\begin{cases}
dX(t)=\left[X(t)+\frac{1}{2}u_1+\frac{1}{2}u_2\right]dt+X(t)dB(t), \quad t\in[0,1],\\
X(0)=\xi=1,
\end{cases}
\end{equation}
and the cost functionals with indefinite coefficients
\begin{equation}\label{61}
\begin{cases}
J_1(X,u_1,u_2)=\mathbb{E}\left[{\int_0^1}4u_1^2(t)-2X^2(t)dt+X^2(T)\right],\\
J_2(X,u_1,u_2)=\mathbb{E}\left[{\int_0^1}2x^2(t)+4u_2^2(t)dt-X^2(T)\right]
\end{cases}
\end{equation}
for the follower and the leader, respectively. Then, applying the It\^{o} formula to $\widetilde{X}^2$, $|\widetilde{X}^*|^2$ and $\overline{\psi}_1\widetilde{X}^*$ (see \eqref{7} and \eqref{11}), one has
\begin{equation*}
\begin{cases}
\|\widetilde{X}_T\|^2_{\mathcal{L}^2_{\mathcal{F}_{T}}(\Omega,\mathbb{R})}=3\|\widetilde{X}\|^2_{\mathcal{L}^2_{\mathcal{F}}([s,T],\mathbb{R})}+\langle \widetilde{X},\widetilde{u}_1\rangle_{\mathcal{L}^2_{\mathcal{F}}([s,T],\mathbb{R})},\\
0=\langle -\overline{\psi}_1+\frac{1}{2}\widetilde{u}_2,\overline{\psi}_1\rangle_{\mathcal{L}^2_{\mathcal{F}}([s,T],\mathbb{R})}+\langle \frac{1}{2}\widetilde{u}_2,\widetilde{X}^*\rangle_{\mathcal{L}^2_{\mathcal{F}}([s,T],\mathbb{R})},\\
\|\widetilde{X}^*_T\|^2_{\mathcal{L}^2_{\mathcal{F}_{T}}(\Omega,\mathbb{R})}=2\langle\widetilde{X}^*,-\overline{\psi}_1+\frac{1}{2}\widetilde{u}_2\rangle_{\mathcal{L}^2_{\mathcal{F}}([s,T],\mathbb{R})}+
\|\widetilde{X}^*\|^2_{\mathcal{L}^2_{\mathcal{F}}([s,T],\mathbb{R})}.
\end{cases}
\end{equation*}
Then, we have
\begin{align*}
\mathbb{E}\left[{\int_0^1}4|\widetilde{u}_1(t)|^2-2|\widetilde{X}(t)|^2 dt+|\widetilde{X}(T)|^2\right]
=&4\|\widetilde{u}_1\|^2_{\mathcal{L}^2_{\mathcal{F}}([s,T],\mathbb{R})}+\|\widetilde{X}\|^2_{\mathcal{L}^2_{\mathcal{F}}([s,T],\mathbb{R})}+\langle \widetilde{X},\widetilde{u}_1\rangle_{\mathcal{L}^2_{\mathcal{F}}([s,T],\mathbb{R})}\\
\geq&\frac{15}{4}\|\widetilde{u}_1\|^2_{\mathcal{L}^2_{\mathcal{F}}([s,T],\mathbb{R})}
\end{align*}
and
\begin{align*}
&\mathbb{E}\left[{\int_0^1}4|\widetilde{u}_2(t)|^2+2|\widetilde{X}^*(t)|^2dt-|\widetilde{X}^*(T)|^2\right]\\
=&4\|\widetilde{u}_2\|^2_{\mathcal{L}^2_{\mathcal{F}}([s,T],\mathbb{R})}-2\langle\widetilde{X}^*,-\overline{\psi}_1+\frac{1}{2}\widetilde{u}_2\rangle_{\mathcal{L}^2_{\mathcal{F}}([s,T],\mathbb{R})}+
\|\widetilde{X}^*\|^2_{\mathcal{L}^2_{\mathcal{F}}([s,T],\mathbb{R})}\\
=&4\|\widetilde{u}_2\|^2_{\mathcal{L}^2_{\mathcal{F}}([s,T],\mathbb{R})}+2\langle\widetilde{X}^*,\overline{\psi}_1\rangle_{\mathcal{L}^2_{\mathcal{F}}([s,T],\mathbb{R})}+
\|\widetilde{X}^*\|^2_{\mathcal{L}^2_{\mathcal{F}}([s,T],\mathbb{R})}\\
&\mbox{}+2\|\overline{\psi}_1\|^2_{\mathcal{L}^2_{\mathcal{F}}([s,T],\mathbb{R})}-2\langle\widetilde{X}^*,\widetilde{u}_2\rangle_{\mathcal{L}^2_{\mathcal{F}}([s,T],\mathbb{R})}
-\langle\widetilde{u}_2,\overline{\psi}_1\rangle_{\mathcal{L}^2_{\mathcal{F}}([s,T],\mathbb{R})}\\
=&\|\widetilde{X}^*+\overline{\psi}_1-\widetilde{u}_2\|^2_{\mathcal{L}^2_{\mathcal{F}}([s,T],\mathbb{R})}+\|\overline{\psi}_1+\frac{1}{2}\widetilde{u}_2\|^2_{\mathcal{L}^2_{\mathcal{F}}([s,T],\mathbb{R})}
+\frac{9}{4}\|\widetilde{u}_2\|^2_{\mathcal{L}^2_{\mathcal{F}}([s,T],\mathbb{R})}\\
\geq&\frac{9}{4}\|\widetilde{u}_2\|^2_{\mathcal{L}^2_{\mathcal{F}}([s,T],\mathbb{R})}.
\end{align*}
Thus conditions $(\mathbf{H1})$-$(\mathbf{H2})$ can be satisfied. Moreover, one can check  that $(\phi_1,\widetilde{\phi}_1)=(1,0)$ and
$$(\phi_2,\widetilde{\phi}_2)=\left(\begin{bmatrix}
-1&0\\
0&1-e^{-t}
\end{bmatrix},
\begin{bmatrix}
0&0\\
0&0
\end{bmatrix}\right)$$
are unique solutions to \eqref{9} and \eqref{+5}, respectively, and so $(\mathbf{H3})$ can also be satisfied.

We now focus on the following two cases for the equality constraint and the inequality constraint, respectively.
\begin{example}\label{+11}
For the Stackelberg game described by  \eqref{60} and \eqref{61}, we further consider the following equality constraint:
$$
\mathbb{E}\left[\int_0^1(1-e^{-t})X^*(t)+2u_1^*(t)+2e^{-t}u_2(t)dt\right]=a,
$$
where $a$ is a parameter. Obviously, the constraint meets the condition $(\mathbf{H4})$. By \eqref{+6} and \eqref{18}, one has $(\psi_2,\widetilde{\psi}_2)=(0^{2\times1},0^{2\times1})$, $S_0=\frac{1-e^{-2}}{2}>0$  and
$$\widetilde{J}(\lambda,u_2^*(\lambda))=-1-2a\lambda-\frac{1-e^{-2}}{2}\lambda^2.$$
Thus according to Theorem \ref{44},  the   feedback Stackelberg equilibrium $(\widehat{u}_1,\widehat{u}_2)$ is given by
\begin{equation*}
\begin{cases}
\widehat{u}_1=\frac{1}{2}\begin{bmatrix}
1&1-e^{-t}
\end{bmatrix}\widehat{Z},\\
\widehat{u}_2=\frac{1}{2}\begin{bmatrix}
1&1
\end{bmatrix}\widehat{Z}-\frac{e^{-t}}{2}\lambda^*
\end{cases}
\end{equation*}
with $\lambda^*=\frac{2a}{e^{-2}-1}$, where the state $\widehat{Z}$ satisfies the SDE
\begin{equation*}
\begin{cases}
d\widehat{Z}=\left[\begin{bmatrix}
-1&2-e^{-t}\\
-1&0
\end{bmatrix}\widehat{Z}+
\begin{bmatrix}
\frac{2a e^{-t}}{1-e^{-2}}\\
\frac{2a}{e^{-2}-1}
\end{bmatrix}\right]dt+\widehat{Z}dB_t, \quad t\in[0,1]\\
\widehat{Z}(0)=
\begin{bmatrix}
1\\
0
\end{bmatrix}.
\end{cases}
\end{equation*}

\begin{figure}[htbp]
\centering
\includegraphics[height=7cm, width=10cm]{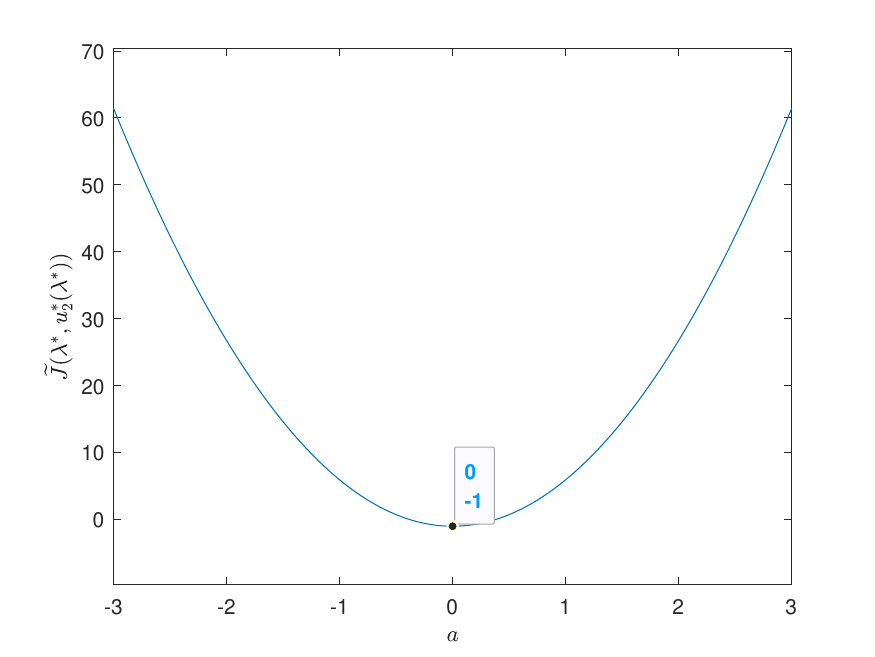}
\caption{Change in $\widetilde{J}(\lambda,u_2^*(\lambda))$ with $a$.}
\label{Fig.1}
\end{figure}
Figure 1 shows that the value function $\widetilde{J}(\lambda^*,u_2^*(\lambda^*))$ is strictly concave with respect to $a$ and reaches its minimum when $a=0$, i.e., the constraint has no impact on both $\widehat{Z}$ and the  feedback Stackelberg equilibrium.
\end{example}

\begin{example}
For the Stackelberg game described by \eqref{60} and \eqref{61}, we further consider the following equality constraint:
$$
\mathbb{E}\left[\int_0^1(1-e^{-t})X^*(t)+2u_1^*(t)+2e^{-t}u_2(t)dt\right]\leq a,
$$
where $a$ is a parameter. Obviously, the constraint meets the condition $(\mathbf{H4})$. Similar as Example \ref{+11}, we know that
$$\widetilde{J}(\lambda,u_2^*(\lambda))=-1-2a\lambda-\frac{1-e^{-2}}{2}\lambda^2$$
and  the   feedback Stackelberg equilibrium  $(\widehat{u}_1,\widehat{u}_2)$ is given by
\begin{equation*}
\begin{cases}
\widehat{u}_1=\frac{1}{2}\begin{bmatrix}
1&1-e^{-t}
\end{bmatrix}\widehat{Z},\\
\widehat{u}_2=\frac{1}{2}\begin{bmatrix}
1&1
\end{bmatrix}\widehat{Z}-\frac{e^{-t}}{2}\lambda^*
\end{cases}
\end{equation*}
with
$$\lambda^*=\frac{2a}{e^{-2}-1}\vee0=\max\left\{\frac{2a}{e^{-2}-1},0\right\},$$
where the state $\widehat{Z}$ satisfies the SDE
\begin{equation*}
\begin{cases}
d\widehat{Z}=\left[\begin{bmatrix}
-1&2-e^{-t}\\
-1&0
\end{bmatrix}\widehat{Z}+
\begin{bmatrix}
-e^{-t}(\frac{2a}{e^{-2}-1}\vee0)\\
\frac{2a}{e^{-2}-1}\vee0
\end{bmatrix}\right]dt+\widehat{Z}dB_t, \quad t\in[0,1]\\
\widehat{Z}(0)=
\begin{bmatrix}
1\\
0
\end{bmatrix}.
\end{cases}
\end{equation*}
\begin{figure}[htbp]
\centering
\includegraphics[height=7cm, width=10cm]{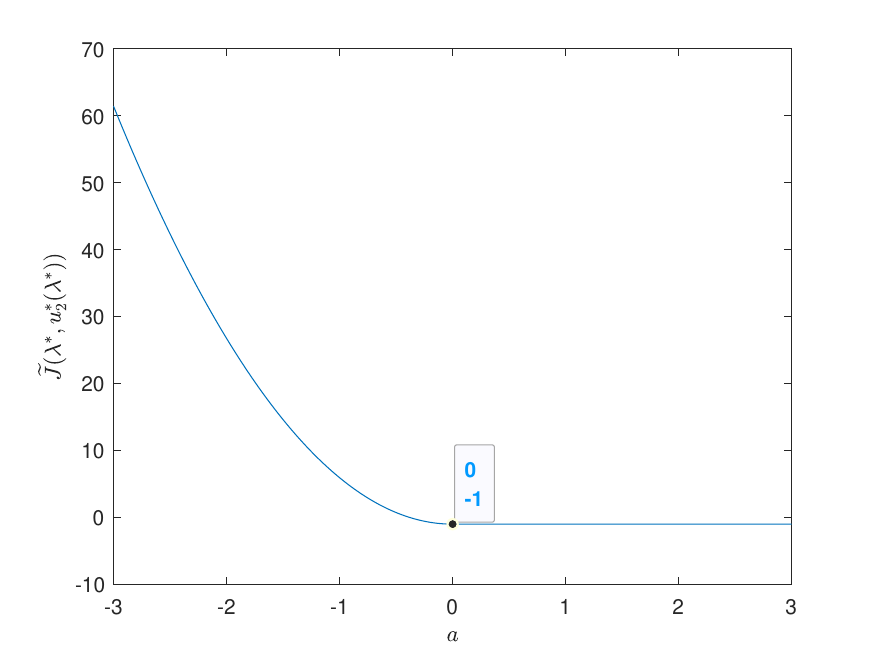}
\caption{Change in $\widetilde{J}(\lambda,u_2^*(\lambda))$ with $a$.}
\label{Fig.2}
\end{figure}
Figure 2 shows that the value function $\widetilde{J}(\lambda^*,u_2^*(\lambda^*))$ is gradually reducing with the growth of $a$, and achieves its minimum when $a\leq0$, which indicates that strengthening the constraints results in a growth of the cost.
\end{example}

\section{Conclusions}
This paper is devoted to studying the LQ-SSDG-AC. By using the Pontryagin maximum principle  and constructing SREs, the feedback optimal strategies are obtained for the follower's problem and the leader's relaxed problem in sequence. Then by employing Lagrangian duality theory, the dual problem of the  leader's problem is obtained, and the state feedback form of the open-loop  optimal strategy is given for the dual problem. After that, the strong duality of the dual problem is obtained under the Slater condition, and a new positive definite condition is proposed for ensuring the uniqueness of solutions to the leader's problem. In addition, the KKT condition is established  for solving the feedback Stackelberg equilibrium of the LQ-SSDG-AC. Finally, two non-degenerate  examples with indefinite coefficients are provided to illustrate the effectiveness of the main results.

It is worth mentioning that some extensions of our problem formulation, including the LQ-SSDG-AC with jumps/partial information \cite{moon2021linear, moon2023linear, xiang2024stochastic, zheng2022linear}, promise to be interesting and important topics. Moreover, the positive definiteness of $S_0$ given  by \eqref{+12} also calls for further research. We plan to address these problems as we continue our research.

\bibliographystyle{plain}
\bibliography{References}

\end{document}